\documentclass[11pt]{amsart}
\usepackage{amsthm,amsfonts,amsmath,amssymb,latexsym,epsfig,mathrsfs,yfonts,marvosym}
\usepackage{amscd}
\usepackage{graphicx}
\usepackage{url}
\usepackage[colorlinks=true]{hyperref}
\usepackage[T1]{fontenc}
\usepackage[small,nohug,heads=littlevee]{diagrams}
\usepackage{centernot}
\usepackage{mathtools}
\usepackage{stmaryrd}
\usepackage[utf8]{inputenc}
\usepackage[english]{babel}
\usepackage[square,numbers]{natbib}
\usepackage{color}
\newtheorem {thm}{Theorem}[section]
\newtheorem {thmA}{Theorem}[]
\newtheorem {lemma} {Lemma} [section]
\newtheorem {prop}[lemma]{Proposition}

\theoremstyle{definition}
\newtheorem{Definition}[lemma]{Definition}
\newtheorem{remark}{Remark}
\newtheorem{ex}[lemma]{Example}

\newtheorem{Question}[lemma]{Question}

\newtheorem{defn}[lemma]{Definition}
\newtheorem{conjecture}{Conjecture}

\newcounter{numl}

\newcommand{\labelnuml}{\textup{(\roman{numl})}}
\newenvironment{numlist}{\begin{list}{\labelnuml}%
{\usecounter{numl}\setlength{\leftmargin}{0pt}%
\setlength{\itemindent}{2\parindent}%
\setlength{\itemsep}{\smallskipamount}\def
\makelabel ##1{\hss \llap {\upshape ##1}}}}{\end{list}}

\newcommand{\be}{\boldsymbol b}
\newcommand{\z}{\zeta}
\newcommand{\R}{{\mathbb R}}
\newcommand{\C}{{\mathbb C}}

\newcommand{\Z}{{\mathbb Z}}
\newcommand{\N}{{\mathbb N}}
\newcommand{\Q}{{\mathbb Q}}
\newcommand{\T}{{\mathbb T}}

\newcommand{\cA}{{\mathcal A}}

\newcommand{\cL}{{\mathcal L}}

\newcommand{\cO}{{\mathcal O}}

\newcommand{\Sph}{\mathbb S}

\newcommand{\symprod}{\mathbin{\raise1pt\hbox{$\scriptstyle\bigcirc$}}}

\newcommand{\p}{\mathrm p}
\newcommand{\bb}{\boldsymbol a}
\newcommand{\tor}{\mathfrak{t}}

\newcommand{\Mpc}{p_{\mathrm{c}}}

\newcommand{\Ref}[1]{(\ref{#1})}
\newcommand{\wht}{\widehat}

\newcommand{\lb}[1]{\label{#1}}
\newcommand{\nab}[1][]{\ensuremath{\mathrm{\nabla}{#1}}}
\newcommand{\om}{\omega}
\newcommand{\fr}{\frac}
\newcommand{\sig}{\sigma}

\begin{document}

\title[Weighted extremal metrics on projective bundles]{Weighted extremal K\"ahler metrics and the Einstein--Maxwell geometry of projective bundles}
\author[Vestislav Apostolov, Gideon Maschler and Christina W. T{\o}nnesen-Friedman]{Vestislav Apostolov, Gideon Maschler, Christina W. T{\o}nnesen-Friedman}
\thanks{The first named author was supported in part by an NSERC Discovery Grant. He is grateful to the Institute of Mathematics and Informatics of the Bulgarian Academy of Science where a part of this work was accomplished. The third named author was supported by the grant 422410 from the Simon's Foundation. The authors are grateful to Abdellah Lahdili for his insightful remarks on the manuscript and to David Calderbank for his help in clarifying some aspects of  Appendix A}
\address{Vestislav Apostolov \\ D{\'e}partement de Math{\'e}matiques\\
UQAM\\ C.P. 8888 \\ Succ. Centre-ville \\ Montr{\'e}al (Qu{\'e}bec)} \email{apostolov.vestislav@uqam.ca}
\address{Gideon Maschler \\ Department of Mathematics and Computer Science\\Clark University\\Worcester\\ MA 01610\\ USA } \email{gmaschler@clarku.edu}
\address{Christina W. T{\o}nnesen-Friedman\\ Department of Mathematics\\ Union
College\\ Schenectady\\ New York 12308\\ USA } \email{tonnesec@union.edu}
\date{\today}

\begin{abstract} We study the existence of weighted extremal K\"ahler metrics in the sense of \cite{aclg17,lahdili}  on the total space of an admissible projective bundle over a Hodge K\"ahler manifold of constant scalar curvature. Admissible projective bundles have been defined in  \cite{acgt}, and they  include the projective line bundles~\cite{hwang-singer}  and their blow-downs~\cite{koi-sak1}, thus providing a most general setting for extending the existence theory for extremal K\"ahler metrics pioneered by a seminal construction of Calabi~\cite{calabi}. We obtain a general existence result for weighted extremal metrics on admissible manifolds,  which yields  many new examples of conformally K\"ahler, Einstein--Maxwell metrics of complex dimension $m>2$, thus extending the recent  constructions of \cite{KoTo16,Le15} to higher dimensions. For each admissible K\"ahler class on an admissible projective bundle, we  associate an explicit function of one variable and show that if it is positive on the interval $(-1,1)$, then there exists a weighted extremal K\"ahler  metric in the given class, whereas if it is strictly negative somewhere in $(-1,1)$, there is no  K\"ahler metrics of constant weighted scalar curvature in that class. We also relate the positivity of the function to a notion of weighted K-stability, thus establishing a Yau--Tian--Donaldson type correspondence for the existence of K\"ahler metrics of constant weighted scalar curvature in the rational admissible K\"ahler classes on an admissible projective bundle.
Weighted extremal orthotoric metrics are examined in an appendix.\end{abstract}

\maketitle

\tableofcontents

\section{Introduction}

\subsection{Motivation and Background}
There has been a great deal of interest recently in studying  the relation between the existence of extremal K\"ahler metrics, especially constant scalar curvature K\"ahler metrics (CSCK), on a compact complex $m$-dimensional manifold and various notions of algebro-geometric stability: this is the so-called Yan--Tian--Donaldson (YTD) conjecture.

The YTD conjecture takes its origins in a formal GIT picture in which the formal (infinite dimensional) K\"ahler manifold is the Fr\'echet space  ${\mathcal K}(M, \omega)$ of all  complex structures  $J$ on $M$, compatible with a given symplectic form $\omega$,  whereas the formal complex structure ${\bf J}$ and K\"ahler form $\boldsymbol{\Omega}$ on ${\mathcal K}(M,\omega)$ are  defined by
\begin{equation}\label{kahler-structure}
{\bf J}_J(\dot J)= J \dot J, \ \ \ {\boldsymbol \Omega}_J({\dot J}_1, {\dot J}_2) = \int_M {\rm tr} \Big(J {\dot J}_1 {\dot J}_2\big) \frac{\omega^m}{m!},
\end{equation}
where $\dot J, {\dot J}_1, {\dot J}_2 \in \Gamma ({\rm End}(TM))$ are tangent vectors at $J\in {\mathcal K}(M, \omega)$,   viewed as smooth fields of endomorphisms of $TM$ which anti-commute with $J$. The main observation due to Fujiki~\cite{fujiki} and Donaldson~\cite{donaldson} is that the infinite dimensional Lie group ${\rm Ham}(M, \omega)$ of hamiltonian transformations of $(M, \omega)$ acts on ${\mathcal K}(M, \omega)$ in a hamiltonian way, with moment map
\begin{equation}\label{moment map}
 {\boldsymbol \mu}(J)= {\mathring {Scal}}(g_J)={Scal}(g_J) - {\underline S}(\omega),
 \end{equation}
 where $Scal(g_J)$ is the scalar curvature of the K\"ahler metric $g_J(\cdot, \cdot):= \omega(\cdot, J\cdot)$, ${\underline S}(\omega)$ is its average,  and the momentum map ${\boldsymbol \mu}(J)$ at $J$ is viewed as an element of the Lie algebra ${\rm Lie}({\rm Ham}(M, \omega)) \cong (C^\infty(M))_0$ of smooth functions with zero mean with respect to the measure $\omega^m/m!$, via the ${\rm Ad}$-invariant inner product
 \begin{equation}\label{inner-product-classic}
\langle \varphi_1, \varphi_2 \rangle = \int_M  \varphi_1 \varphi_2 \frac{\omega^m}{m!}.
\end{equation}
Thus, finding CSCK metrics in ${\mathcal K}(M, \omega)$ corresponds to finding zeroes of the momentum map ${\boldsymbol \mu}$ whereas finding extremal K\"ahler metrics in ${\mathcal K}(M, \omega)$ corresponds to finding critical points of the Calabi functional ${\bf Cal}(J)=||{\boldsymbol \mu}(J)||^2= \int_M (Scal(g_J)-{\underline S})^2 \omega^m/m!$.

\smallskip
It is natural to expect that  for other  geometric problems in K\"ahler geometry, for which we have a similar formal GIT interpretation, a  suitable YTD type correspondence holds true.   This is the case for a class of problems,   motivated in \cite{am} and further developed in \cite{lahdili3},    which can be fit into the the above GIT picture by fixing a torus $\T \subset {\rm Ham}(M, \omega)$ with momentum map $z : M \to \tor^*$, and image  a compact polytope $P \subset \tor^*$ (here $\tor$ denotes the Lie algebra of $\T$ and $\tor^*$ its dual) and  two positive smooth functions $u, v : P \to \R$. Then,  the space ${\mathcal K}^{\T}(M,\omega)$ of $\T$-invariant $\omega$-compatible complex structures  on $M$ caries a formal K\"ahler structure $({\bf J}, {\boldsymbol \Omega}^u)$  with
\begin{equation}\label{symplectic-modified}
{\boldsymbol \Omega}^u_{J}({\dot J}_1, {\dot J}_2) = \int_M {\rm tr} \Big(J {\dot J}_1 {\dot J}_2\big) (u\circ z) \frac{\omega^m}{m!},
\end{equation}
which is invariant under the action of the group ${\rm Ham}^{\T}(M, \omega)$ of hamiltonian transformation commuting with $\T$. Furthermore,  the action of ${\rm Ham}^{\T}(M, \omega)$  is hamiltonian. By considering the ${\rm Ad}$-invariant inner product
\begin{equation}\label{inner-product-modified}
\langle \varphi_1, \varphi_2 \rangle_v = \int_M  \varphi_1 \varphi_2  (v\circ z) \frac{\omega^m}{m!},
\end{equation}
on the Lie algebra ${\rm Lie}\big({\rm  Ham}(M, \omega)^{\T}\big) \cong \big(C^{\infty}(M)\big)^{\T}_{0,v}$ of $\T$-invariant smooth functions of zero mean with respect to the measure $(v\circ z)\omega^m/m!$,  one identifies linear functionals on ${\rm Lie}\big({\rm  Ham}(M, \omega)^{\T}\big)$ with $\T$-invariant smooth functions,  and
denote by
\begin{equation}\label{moment-map-modified}
{\mathring {Scal}}_{u,v, \T}(g_J):= {\boldsymbol \mu}_{u,v, \T}(J)
\end{equation}
the corresponding moment map. One thus expects that the problem of finding $\T$-invariant K\"ahler metrics in a given K\"ahler class, for which ${\mathring {Scal}}_{u,v, \T}(g_J)= 0$ or, more generally, for which  ${\rm grad}_{g_J}({\mathring {Scal}}_{u,v, \T}(g_J))$ is a holomorphic vector field,   shares many common features  with the well-established theory of CSCK  and extremal  K\"ahler metrics. We shall refer to such special K\"ahler metrics as  {\it $(u,v)$-CSCK} and {\it $(u,v)$-extremal} K\"ahler metrics, respectively.

\smallskip There are a number of geometric situations which reduce to the above formal GIT setting for particular choices of the functions $u$ and $v$. These  include the problem of the existence of admissible CSCK metrics on rigid semi-simple toric bundles \cite{ACGT4}, and on manifolds with free multiplicity~\cite{donaldson-survey}.

\bigskip  In this paper, we shall focus on the geometric problem introduced in \cite{am} and \cite{aclg17}, which fits into the above context by choosing  $\xi \in \tor$ and $\bb>0$ such that $\pi_{\xi}(P) + \bb >0$ (where $\pi_{\xi} : \tor^* \to \tor^*_{\xi} = \R$  is the projection dual to the inclusion $\R \cdot \xi \subset \tor$),  and  letting  $u:= (\pi_{\xi}+\bb)^{-(\p+1)}$ and $v:=(\pi_{\xi} +\bb)^{1-\p}$ for $\p\in \R$. As observed in \cite{aclg17,lahdili2}, the corresponding momentum map \eqref{moment-map-modified} then becomes (up to an additive constant depending only on the deRham class $[\omega]$)
\begin{equation}\label{scalh1}
Scal_{f,\p}(g)= f^2 Scal(g)  -2(\p-1) f\Delta_g f - \p(\p-1)|df|^2_g,
\end{equation}
where $f:= \langle z, \xi \rangle + \bb$ is a positive Killing potential of the K\"ahler metric $g$.  The smooth function $Scal_{f,\p}(g)$  is referred to as {\it  the $(f, \p)$-scalar curvature} of $g$, and we are interested in finding K\"ahler metrics of constant $(f, \p)$-scalar curvature, which we call {\it $(f,\p)$-CSCK metrics} or, more generally,   K\"ahler metrics for which $Scal_{f,\p}(g)$ is a Killing potential, which we refer to as {\it $(f,\p)$-extremal K\"ahler metrics}.  The case $\p=2m$ has been studied in \cite{ambitoric1,am,futakiono17a,futakiono17b, lahdili, lahdili2, LeJGP, Le15, KoTo16}, having the geometric meaning that $Scal_{f,2m}(g)$ computes  the scalar curvature of the Hermitian metric $h = f^{-2} g$. Thus, a K\"ahler metric $g$ for which $Scal_{f,2m}(g)$ is constant corresponds to a Hermitian metric $h=f^{-2}g$  which is {\it conformally K\"ahler,  Einstein--Maxwell}, see \cite{Le08,ambitoric1,am}. Another geometrically interesting case  is when $\p=m+2,$  which is related to the study of Levi--K\"ahler quotients~\cite{aclg17} and extremal Sasaki metrics~\cite{ac}.

\subsection{Main Results} The main results of this paper concern the existence  and obstruction theory of $(f, \p)$-extremal K\"ahler metrics  and $(f, \p)$-CSCK metrics on certain holomorphic projective bundles of the form $M={P}(E_0 \oplus E_{\infty}) \to S,$ where $E_0$ and $E_{\infty}$ are projectively flat bundles (of arbitrary rank) over a compact CSCK manifold $S$. Such complex manifolds, introduced and studied  in \cite{acgt} are called {\it admissible}. Any admissible manifold $M$ is  endowed with a natural $\Sph^1$-action,  induced by  the natural $\Sph^1$-action  by fibre-wise multiplications on the bundle $E_0$. Following \cite{acgt}, $M$ admits  a family  of $\Sph^1$-invariant symplectic structures $\omega$ (called {\it admissible}) with associated momentum image $z(M)= [-1,1] \subset \R$.
With this normalization, for any real constants $\bb>1$ and $\p$, we consider the problem of finding an $\omega$-compatible $(z+\bb, \p)$-CSCK metric or, more generally, an $\omega$-compatible K\"ahler metric for which $Scal_{(z+\bb), \p}$ is  a Killing potential (called a $(z+\bb, \p)$-extremal K\"ahler metric). The construction of $(z+\bb, \p)$-extremal K\"ahler metrics on admissible complex manifolds naturally extends the well-known constructions, going back to Calabi~\cite{calabi},  of  explicit extremal K\"ahler metrics on $\C P^1$-bundles over a CSCK base and their blow-downs (see \cite{acgt} and the references therein) and involves a smooth function of a single variable, defined on the momentum image $[-1,1] \subset \R$.  More precisely, we show (see Proposition~\ref{emsolution}) that  for {\it any} admissible K\"ahler class $\Omega=[\omega]$ on the admissible manifold $M$,  there exists an explicitly defined  smooth function  $F_{\Omega, \bb, \p}(z)$  which gives rise to an explicit $(z+\bb, \p)$-extremal K\"ahler metric in $\Omega$ (given by the Calabi ansatz),   provided  that  $F_{\Omega, \bb, \p}(z)$ is strictly positive on $(-1,1)$. This gives  a {\it sufficient condition} for $\Omega$ to admit a $(z+\bb, \p)$-extremal K\"ahler metric, and  leads to many new examples.  A sample  is provided by the following
\begin{thmA}(see Theorem~\ref{fextremalexistence} below) Let  $M={P}(E_0 \oplus E_{\infty})\to S$ be an admissible manifold over a compact CSCK manifold  $S$. Then, for any  $\bb >1$ and  $\p\in \R$, $M$ admits (admissible) $(z+\bb, \p)$-extremal K\"ahler metrics in some admissible K\"ahler classes. If, moreover, the K\"ahler manifold $S$  is a local product of nonnegative CSCK metrics, then  for any  $\bb >1$ and  $\p\in \R$, every admissible K\"ahler class contains an (admissible) $(z+\bb,\p)$-extremal K\"ahler metric.
\end{thmA}

Similarly to the extremal case studied~\cite{acgt}, the smooth function $F_{\Omega, \bb, \p}(z)$ also allows one to {\it compute} the  vanishing of  the $(z+\bb,\p)$-Futaki invariant associated to the admissible K\"ahler class $\Omega$, which in turn is the obstruction found in \cite{am,lahdili} for a $(z+\bb,\p)$-extremal K\"ahler metric in $\Omega$ to be actually $(z+\bb, \p)$-CSCK (see Proposition~\ref{futaki}). Specializing to the case $\p=2m$, this leads to many new explicit examples of conformally K\"ahler, Einstein--Maxwell metrics  (see Propositions~\ref{O(-1,1)} and \ref{hodge}), thus extending the constructions in \cite{LeJGP,Le15,futakiono17a,KoTo16} to higher dimensions.

\smallskip
Another aspect of theory in \cite{acgt}, which we partially extend in this paper to the $(z+\bb, \p)$-extremal case, consists of showing that the positivity of the function $F_{\Omega, \bb, \p}(z)$  on the interval $(-1,1)$ is also  a {\it necessary condition} for the  existence of a $(z+\bb, \p)$-extremal K\"ahler metric in $\Omega$. This was achieved in \cite{acgt} by using the following  two deep results concerning extremal K\"ahler manifolds: (a) the boundedness of the relative Mabuchi energy,  and,  (b) the uniqueness of the extremal K\"ahler metrics modulo the action of the automorphism group of $M$, see \cite{bb,ct}. Neither of these results is yet available in the general $(f, \p)$-extremal setting, but A. Lahdili~\cite{lahdili2} has recently established the analogue of (a) in the $(f, \p)$-extremal case, assuming that $\Omega$ is  rational and  that    the corresponding  $(f,\p)$-Futaki invariant vanishes.  Expressing the  relative $(K, \bb, \p)$-Mabuchi energy (Definition~\ref{d:mabuchi}) of an admissible manifolds in terms of the function $F_{\Omega, \bb, \p}(z)$ (Propositions~\ref{mabuchi} and \ref{Lahdili}), we thus obtain
\begin{thmA}\label{A}(see Theorem~\ref{lahdili2}) Let $M={P}(E_0 \oplus E_{\infty})\to S$  be an admissible K\"ahler manifold over a compact CSCK manifold $S$. Suppose that the admissble K\"ahler class  $\Omega$ is a positive multiple of an element  in  $H^{2}(M, \Z)$, and that the corresponding  smooth function $F_{\Omega, \bb, \p}(z)$ is strictly negative somewhere on $(-1,1)$. Then $\Omega$ does not admit a K\"ahler metric of constant $(z+\bb,\p)$-scalar curvature.
\end{thmA}
We note that  in the case when  $M$ is a geometrically ruled complex surface over a compact curve of genus $\ge 2$ and $\p=4$, the above result was further strengthened  in \cite{lahdili2}, where  the rationality assumption on $\Omega$ was dropped and the non-existence  of a $(z+\bb, 4)$-CSCK metric was shown to also hold if  $F_{\Omega, \bb, 4}(z)$  vanishes somewhere on $(-1,1)$.

\bigskip
The final and perhaps most original theme of this paper concerns the link of the above existence and obstruction results with the algebro-geometric notion of (relative)  $(\hat \beta, \p)$-K-stability,  proposed in \cite{am}: Given a polarized compact smooth projective variety  $(M, L)$  and a  quasi-periodic holomorphic vector field $\hat K \in {\rm Lie}({\rm Aut}(M, L))$,  the considerations in \cite[Sect. 7.1]{am} lead to the definition of a  (relative) $(\hat K, \p)$-Donaldson--Futaki  invariant associated to a certain sub-class of test configurations compatible with $(M, L, \hat K)$, see Section~\ref{beta-futaki} below. However, one caveat of  this definition is that it involves transcendental  quantities when $\p$  is not a negative integer, leading to difficulties reminiscent to the ones involved in the definition of the $L^p$-norm of a test configuration for positive real values of $p$, see the discussion at the end of \cite{donaldson06}. Indeed, to the best of our knowledge, no direct link  is established so far between the existence of an  $(f, \p)$-CSCK metric in $c_1(L)$ and the sign of the $(\hat K, \p)$-Donaldson--Futaki  invariant of a  compatible normal test configuration,  beyond the toric context considered in \cite{am}.

Thus motivated, on an admissible manifold $M=P(E_0\oplus E_{\infty})\to S$ polarized by an ample line bundle $L$ whose Chern class is  a multiple of an  admissible K\"ahler class $\Omega$, we endow the total space $L$ with a quasi-periodic real holomorphic vector field $\hat \beta_{\bb}$,  obtained by lifting the generator of the $\Sph^1$-action on $M$ with the help of its momentum map $(z+\bb)$. We further consider the $1$-parameter family of ($\hat \beta_{\bb}$-compatible) test-configurations corresponding to the degeneration to the normal cone of the infinity section $P(0 \oplus E_{\infty}) \subset M$, see \cite{RT}, and compute the corresponding relative $(\hat \beta_{\bb}, \p)$-Donaldson--Futaki  invariant,  by adapting the arguments of \cite{sz} and \cite{acgt} to our setting. The upshot is that  the relative $(\hat \beta_{\bb}, \p)$-Donaldson--Futaki  invariant of such test-configurations, which we call {\it admissible},  is given by a  positive multiple of the function $F_{\Omega, \bb, \p}(\z)$ alluded to above, where $\z\in (-1,1)\cap \Q$ parametrizes the admissible test configurations.  We thus say that on $(M,L)$, the {\em relative} version of $(\hat \beta_{\bb}, \p)$-K-semistability/$(\hat \beta_{\bb},\p)$-K-stability/analytically $(\hat \beta_{\bb},\p)$-K-stability holds on admissible test configurations if $F_{\Omega, \bb, \p} (\z) \ge 0$ on $(-1,1)$/ $F_{\Omega, \bb, \p} (\z) > 0$ on $(-1,1)\cap \Q$/ $F_{\Omega, \bb, \p} (\z) > 0$ on $(-1,1)$, respectively. On the other hand, $(M,L)$ is  said to be $(\hat \beta_{\bb}, \p)$-K-semistable/$(\hat \beta_{\bb},\p)$-K-stable/analytically $(\hat \beta_{\bb},\p)$-K-stable on admissible test configurations if the above holds and, additionally, the $(z+\bb, \p)$-Futaki invariant of the admissible K\"ahler class $\Omega$ vanishes. Our results then can be summarized in  the following YTD type correspondence.
\begin{thmA} (see Theorem~\ref{thm:admissible-stable}) Let $M = P(E_0 \oplus E_{\infty}) \to S$ be an admissible manifold, and $L$ an ample holomorphic  line bundle  on $M,$ which defines, up to a positive multiple, an admissible K\"ahler class $\Omega$.
\begin{enumerate}
\item[$\bullet$] If for some $\bb > 1$,  $(M, L)$ is  analytically relative $(\hat \beta_{\bb},\p)$-K-stable (resp. analytically $(\hat \beta_{\bb},\p)$-K-stable) with respect to admissible test configurations, then there exists an admissible $(z + \bb, \p)$-extremal K\"ahler metric in  $\Omega$ (resp. an admissible K\"ahler metric of constant $(z + \bb, \p)$-scalar curvature).
\item[$\bullet$] If $\Omega$ admits a K\"ahler metric of constant $(z + \bb, \p)$-scalar curvature, then $(M, L)$ is $(\hat \beta_{\bb}, \p)$-K-semistable with respect to  admissible test configurations.
\end{enumerate}
\end{thmA}
Conjecturally, the second claim should be improved to showing that the existence of a K\"ahler metric of constant $(z + \bb, \p)$-scalar curvature in $\Omega$ implies that $(M, L)$ is  analytically $(\hat \beta_{\bb}, \p)$-stable with respect to admissible test configurations. Using the results in \cite{KoTo16,lahdili2} we are able to establish this in the case of a geometrically ruled complex surface over a curve of genus $\ge 2$ and $\p=4$:
\begin{thmA} (see Theorem~\ref{thm:ruled-classification}) Let $M={P}(\cO \oplus E) \to \Sigma$ be a ruled complex surface, where  $E$ is a line bundle of positive degree over  a complex curve $\Sigma$ of genus $\ge 2$,  $L$ a  polarization of $M$,  which, up to a positive multiple, corresponds to an admissible K\"ahler class $\Omega$.  Then the following conditions are equivalent
\begin{enumerate}
\item[\rm (i)]  $\Omega$ admits a $(z+\bb, 4)$-CSCK metric, or, equivalently, a conformally K\"ahler, Einstein--Maxwell metric;
\item[\rm (ii)] $(M, L)$ is $(\hat \beta_{\bb}, 4)$-K-stable on admissible test configurations;
\item[\rm (iii)] $(M, L)$ is analytically $(\hat \beta_{\bb}, 4)$-K-stable on admissible test configurations.
\end{enumerate}
\end{thmA}

\subsection{Structure of the paper} Section~\ref{review} contains the main technical body of the paper.  After a brief review  of the admissible setting of \cite{acgt}, we summarize in Proposition~\ref{emsolution} the main tools allowing us to  extend the theory of \cite{acgt} from extremal to weighted extremal metrics. With this in place, we compute in Proposition~\ref{futaki} the weighted Futaki invariant, and in Proposition~\ref{mabuchi} the relative weighted Mabuchi  functional, associated to admissible K\"ahler metrics. The upshot is Proposition~\ref{Lahdili} which links the function $F_{\Omega, \bb, \p}(z)$ of Theorem~\ref{A} to the boundedness from below of the weighted Mabuchi functional. In Section~\ref{beta-futaki}, we explore the notion of weighted K-stability proposed in  \cite{am} and compute, using the method of \cite{sz} and \cite{acgt},  the (relative) weighted  Donaldson--Futaki invariant of an admissible test configuration (a special case of the degeneration to the normal cone studied in \cite{RT}).  In Section~\ref{proofs}, we present the proofs of the main results  from the Introduction. In the final Section~\ref{s:EM}, we specialize our existence results to the case of conformally K\"ahler, Einstein--Maxwell metrics. In particular, we obtain a large family of new examples in  any real dimension $\ge 6$, see Example \ref{negative-scal}, Propositions~\ref{O(-1,1)}, \ref{hodge}, and \ref{E-extended}.  In  Section~\ref{YamabeFun}, we show that some of the Einstein--Maxwell metrics (which have positive constant scalar curvature) are not Yamabe minimizers even though they satisfy the strict Aubin--Schoen  inequality. Finally, local existence and uniqueness results for {\em orthotoric} $(f,\p)$-extremal metrics are given in an appendix.

\section{Admissible manifolds and metrics}\label{review}
In this section we  review a well-known construction, originally used by Calabi~\cite{calabi} to construct extremal K\"ahler metrics on the Hirzebruch complex surfaces, and  generalized in many subsequent works~\cite{Guan95,Hwang94,hwang-singer,koi-sak1,sakane,christina1} to the case of  $\C P^1$-bundles over a CSCK base and their blow-downs.  We follow closely the notation in \cite{acgt} which combines in the same formalism the momentum profile construction of \cite{hwang-singer} and the blow-down  construction of \cite{koi-sak1,sakane}.

\subsection{Admissible manifolds}
 Let $(S, g_S, \omega_S)$ be a  compact K\"ahler  manifold covered by the product of simply connected K\"ahler manifolds $(S_a, \pm g_a, \pm \omega_a), a \in \cA \subset \Z^+$, where $\cA=\{1, \ldots, N\}$ is a finite index set and $(\pm g_{a},\pm
\omega_{a})$
are the K\"ahler structures with the usual sign ambiguity allowing for $g_a$ and $\omega_a$ to be negative definite tensors (in which case we write $(-g_a, -\omega_a)$ for the K\"ahler structure on $S_a$).
In all our applications, each $\pm g_a$ is assumed to have constant scalar curvature, i.e. $(g_S, \omega_S)$ is a constant scalar curvature metric (CSCK) on $S$.
The real dimension of each component is denoted
$2 d_{a}$, while the scalar curvature $Scal(\pm g_a)$ of $\pm g_{a}$ is written as $\pm 2 d_{a} s_{a}$, where $\pm s_{a}$ is the {\it normalized scalar curvature} of $\pm g_a$.

Let  $E_0$, $E_\infty$ be projectively flat
hermitian holomorphic vector bundles over $S$, of ranks $d_{0}+1$ and $d_{\infty} +1$,
respectively,   satisfying the condition
\begin{equation}\label{topological}
c_{1}(E_{\infty})/(d_{\infty}+1) -
c_{1}(E_{0})/(d_{0}+1) = \sum_{a \in\cA} [\omega_{a}/2\pi].
\end{equation}
Then, following \cite{acgt}, the total space of the projectivization
$M=P(E_{0} \oplus E_{\infty})\xmapsto{p} S$ is called an {\it admissible complex manifold}.

 A useful observation is the following diagram of holomorphic  maps
 \begin{equation}\label{blowdown}
\begin{diagram}
\hat M & = {P} \Big({\hat E}_0 \oplus  {\hat E}_{\infty}\Big) &\rTo^{\hat p} & {\hat S} &={\mathbb P}(E_0)\times_S {\mathbb P}(E_{\infty}) \\
\dTo^{b} &                                                                            &                    & \dTo^{\pi} &  \\
M        &= {P}\Big(E_0\oplus E_{\infty}\Big) & \rTo^{p} & S, &
\end{diagram}
\end{equation}
where ${\hat E}_0=\cO(-1)_{{P}(E_0)}$ and ${\hat E}_{\infty}=  \cO(-1)_{{P}(E_{\infty})}$ are line bundles over $\hat S$. The latter in turn  is a CSCK manifold covered by the K\"ahler product $\C P^{d_0} \times \Big( \prod_{a\in \cA} S_a \Big) \times \C P^{d_{\infty}}$ where $d_0= {\rm rk}(E_0)-1$ and $d_{\infty}={\rm rk}(E_{\infty})-1$. Note that the complex dimension of $M$ is $m=\sum_{a \in \hat{\cA}}d_a+1$.

Let $\hat{\cA} \subset \mathbb{N}\cup\infty$
be the extended index set defined as follows:
\begin{itemize}
\item $\hat{\cA} = \cA$, if $d_{0}=d_{\infty}=0$.
\item $\hat{\cA} =\cA \cup \{ 0 \}$, if $d_{0}>0$ and $d_{\infty}=0$.
\item $\hat{\cA} =\cA \cup \{ \infty \}$, if $d_{0}=0$ and $d_{\infty} >0$.
\item $\hat{\cA} =\cA \cup \{ 0 \} \cup \{ \infty \}$, if $d_{0} > 0$
and $d_{\infty} >0$.
\end{itemize}
Correspondingly, we consider on $\hat S$ the  CSCK metric covered by the product  of the  K\"ahler metrics $(\pm g_a, \pm \omega_a), a\in \hat \cA$ with $(g_{0}, \omega_{0})$  and $(-g_{\infty},-\omega_{\infty})$ being
the Fubini--Study structures with scalar curvatures
$2d_{0}(d_{0} +1)$  and $2d_{\infty}(d_{\infty}+1)$ on  the factors $\C P^{d_0} $ and $\C P^{d_{\infty}},$ respectively.

We will consider the $\C^*$-action on $M$ (resp. on $\hat M$),  defined by diagonal multiplication on $E_0$ (resp. on ${\hat E}_0$) and denote by $M_0$ (resp. ${\hat M}_0$) the open dense subset of regular points of the action. It is not hard to see that the first vertical map in the diagram is a $\C^*$-biholomorphism from $\hat M_0$ to $M_0$, and is referred to in \cite{acgt} as the {\it blown-down} map. In the sequel,  we shall tacitly identify $M_0$ and $\hat M_0$; in particular, $M_0$  has the structure of a principal $\C^*$-bundle over the (stable) quotient under the $\C^*$-action of $M$ and,  in our case,  using the blow down map $b$ in \eqref{blowdown}, it corresponds to the $\C^*$-bundle over $\hat S$,  obtained from the $\C P^1$-bundle $\hat p: \hat M \to \hat S$ by deleting the divisors ${\hat e}_{\infty}= P(0 \oplus {\hat E}_{\infty})$ and ${\hat e}_{0}=P({\hat E}_0 \oplus 0)$.

\subsection{Admissible metrics}\label{s:admissible-metrics}
A particular type of K\"ahler metric on $M$,  also called
{\it admissible}, will now be described as smooth tensors on $M_0$.

Using the assumption  that $E_0$ and $E_{\infty}$ are projectively flat and \eqref{topological}, there exist hermitian metrics $h_0$ on $E_0$ and $h_{\infty}$ on $E_{\infty}$ whose  respective Chern connections  have curvatures
$\Omega_0 \otimes {\rm Id}_{E_0}$ and  $\Omega_{\infty}\otimes {\rm Id}_{E_{\infty}}$, with $\Omega_0$ and $\Omega_{\infty}$ being $2$-forms on $S$ satisfying (when pull-backed to the universal cover of $S$)
$$\Omega_{\infty} - \Omega_0 = \sum_{a\in \cA} \omega_a.$$
The hermitian metrics $h_0$ and $h_{\infty}$ induce hermitian metrics $\hat h_0$ and $\hat h_{\infty}$ on the line bundles $\hat E_0=\cO(-1)_{P(E_0)}\to \hat S$ and $\hat E_{\infty}=\cO(-1)_{P(E_{\infty})}\to \hat S,$ with curvatures $2$-forms $-\omega_0 + \Omega_0$ and $\omega_{\infty} + \Omega_{\infty}$, respectively.  We denote respectively by $K_{0}$ and $K_{\infty}$ the generators of the circle action on $\hat E_0$ and $\hat E_{\infty},$ and by $r_0$ and $r_{\infty}$ the corresponding fibre-wise norm functions.  Using the Chern connections of $(\hat E_{0}, \hat h_0)$  and $(\hat E_{\infty}, \hat h_{\infty})$, we let $\hat \theta_0$ and $\hat \theta_{\infty}$ be the connection 1-forms defined on the corresponding unitary bundles, i.e. satisfying
\begin{equation*}
\begin{split}
\hat \theta_0(K_0) &=1,  \ \ d\hat \theta_0 = \omega_0 - \Omega_0; \\
\hat  \theta_{\infty} (K_{\infty})&=1,  \ \ d\hat \theta_{\infty} = -\omega_{\infty} - \Omega_{\infty}.
 \end{split}
 \end{equation*}
Thus, the fibres-wise euclidean structures (viewed as tensors on the total spaces $\hat E_0$ and $\hat E_{\infty}$) take the following momentum/angular form
\begin{equation*}
\hat g_0=\frac{{dz_0}\otimes dz_0}{2z_0} + 2z_0 (\hat \theta_0\otimes \hat \theta_0), \ \ \hat g_{\infty}=\frac{dz_{\infty}\otimes dz_{\infty}}{2z_{\infty}} + 2z_{\infty} (\hat \theta_{\infty}\otimes \hat \theta_{\infty}),
\end{equation*}
where  $z_0:= r^2_0/2,  z_{\infty} := r_{\infty}^2/2$ are the fibre-wise momentum coordinates.

For each $a\in \cA$ let $|x_a|<1$  be fixed real numbers such that $x_ag_a$ is positive definite, and  $x_{0}:=1$, $x_{\infty} : = -1$. We then  consider the smooth positive semidefinite tensor  on the total space $\hat E_0 \oplus \hat E_{\infty}$
\begin{equation*}
\hat g =\sum_{a\in\hat{\cA}}\frac{(1+x_a)z_0+ (1-x_a)z_{\infty})}{2x_a}g_a + \hat g_{0} + \hat g_{\infty}.
\end{equation*}
Considering the ``K\"ahler quotient''  for $\hat g$  with respect to the $\Sph^1$-action generated by $K_0 + K_{\infty}$ at the  level set $z_0 + z_{\infty}=2$ on $\hat E_0 \oplus \hat E_{\infty}$, we denote by  $g_c$ the smooth (possibly degenerate) tensor field  induced on $\hat M = P(\hat E_0 \oplus \hat E_{\infty})$ and by $\omega = g_c J_c$ the corresponding smooth $(1,1)$-form, where $J_c$ is the induced (canonical) complex structure on $\hat M$. Letting $z:=(z_0-z_{\infty})/2\in [-1,1]$, the degenerate K\"ahler structure $(g_c, \omega)$ is written  on $\hat M_0$ as:
\begin{equation}\label{g}
g_c=\sum_{a\in\hat{\cA}}\frac{1+x_az}{x_a}g_a+\frac {dz^2}
{\Theta_c (z)}+\Theta_c (z)\theta^2,\quad
\omega = \sum_{a\in\hat{\cA}}\frac{1+x_az}{x_a}\omega_{a} + dz \wedge
\theta,
\end{equation}
where $\Theta_c (z)= 1-z^2$ and $\theta := \hat{\theta}_{0}- \hat{\theta}_{\infty}$ satisfies
\begin{equation}\label{theta}
d\theta = \sum_{a\in \hat \cA} \omega_a.
\end{equation}
We notice that $z$ is the momentum map with respect to $\omega$ of the induced $\Sph^1$-action on $\hat M$ corresponding to multiplication  on $\hat E_0$ or,  equivalently, the $\Sph^1$-action induced by the push forward of $K=(K_0-K_{\infty})/2$ to the quotient space $\hat M$. Thus,  $\hat e_{\infty} = z^{-1}(-1), \hat e_{0} = z^{-1}(1),$ and $\hat M_0 = z^{-1}(-1,1)$. It follows that $(g_c, \omega)$ defines a K\"ahler metric on $\hat M_0$,  which degenerates over $\hat e_0 \cup \hat e_{\infty}$ when $\hat \cA \neq \cA$. Nevertheless,  it is shown in \cite{acgt} that $(g_c, \omega)$ gives rise to a genuine, non-degenerate, smooth K\"ahler metric on $M=P(E_0 \oplus E_{\infty})$  where we identify $M_0$ with $\hat M_0$ via \eqref{blowdown}. Then, $z$ is the momentum map  with respect to $\omega$  of the $\Sph^1$-action on $M$ by multiplication on  $E_0$. The formulae \eqref{g} and \eqref{theta} describe the pull-back of $(g_c, \omega)$ to $\hat M$  via the map $b$  in \eqref{blowdown}.

It was observed in \cite{acgt} (the argument actually goes back to \cite{hwang-singer}) that if instead of $\Theta_c(z)$ we take in \eqref{g} any smooth function $\Theta(z)$ on $[-1,1]$, satisfying
\begin{align}
\label{positivity}
(i)\ \Theta(z) > 0, \quad -1 < z <1,\quad
(ii)\ \Theta(\pm 1) = 0,\quad
(iii)\ \Theta'(\pm 1) = \mp 2.
\end{align}
then the formulea
\begin{equation}\label{gg}
g=\sum_{a\in\hat{\cA}}\frac{1+x_az}{x_a}g_a+\frac {dz^2}
{\Theta (z)}+\Theta (z)\theta^2,\quad
\omega = \sum_{a\in\hat{\cA}}\frac{1+x_az}{x_a}\omega_{a} + dz \wedge
\theta,
\end{equation}
and \eqref{theta} introduce a smooth $\Sph^1$-invariant K\"ahler metric on $M$, compatible with the same symplectic form $\omega$. The corresponding  complex structure  is then given on $\hat M_0=M_0$ by the horizontal lift  of the base complex structure on $\hat S$ (with respect to the chosen Chern connections on $\hat E_0$ and $\hat E_{\infty}$) along with the requirement $Jdz = \Theta \theta$ on the fibres. Such K\"ahler metrics on $M$ are called {\it admissible K\"ahler metrics}, and  we denote by $\mathcal{K}^{\rm adm}(M,\omega)$ the space of all admissible K\"ahler metrics associated to a given choice of $x_a, a\in \cA$. Thus, $\mathcal{K}^{\rm adm}(M,\omega)$ is a Fr\'echet space consisting of all smooth functions $\Theta(z)$ on $[-1,1]$ satisfying \eqref{positivity}.  For fixed values $x_a \in (-1,1)$ (and $g_a$), we let  $\Omega_x :=[\omega]$ be the corresponding deRham class on $M$,  which we refer to as an  {\it admissible K\"ahler class} on $M$.  Note that from Section 1.3 in \cite{acgt}, it follows that $\Omega_x$ is  a positive multiple of an element in $H^2(M,{\mathbb Z})$ precisely when $x_a \in {\mathbb Q}$ for all $a \in \cA$.

 The space $\mathcal{K}^{\rm adm}(M, \omega)$  (associated  to a given data  $(x_a, g_a)$)    can also be equally parametrized by the fibre-wise symplectic potentials $u(z)$, where $u(z)$ is defined up to an affine-linear  function of $z$ by  $u''(z)= \frac{1}{\Theta(z)}$. It is shown in \cite[p. 566]{acgt} that the fibre-wise Legendre transform ${\mathcal T}$ maps  $\mathcal{K}^{\rm adm}(M, \omega)$ to the space $\mathcal{K}(M, J_c, \Omega)=\{\varphi \in C^{\infty}(M) : \omega + dd^c \varphi >0\}$ of  $J_c$-compatible K\"ahler metrics in the class $\Omega_x=[\omega]$,    and has differential given by
\begin{equation} \label{symplectic/complex}
({\bf d} {\mathcal T})_{g}(\dot u) \to -\dot \varphi
\end{equation}

\begin{remark}\label{r:symmetry}
The parametrization of the space  $\mathcal{K}^{\rm adm}(M, \omega)$ of admissible K\"ahler metrics  in terms of $(x_a, g_a), a \in \hat \cA$ and $\Theta(z)$ is not effective. Indeed,  any admissible K\"ahler metric $(g, \omega) \in {\mathcal K}^{\rm adm}(M, \omega)$  of the form \eqref{gg} can be also obtained by changing $g_a$ to $-g_a$ for $a\in \hat A$, $z$ to $-z$ and $\theta$ to $-\theta$. Geometrically, this corresponds simply to changing the r\^oles of  the vector bundles $E_0$ and $E_{\infty}$ or,  equivalently, changing the generator $K$ of the $\Sph^1$-action by multiplications on $E_0$ with $-K$. Notice that $-K$ is the generator of the $\Sph^1$-action on $P(E_0 \oplus E_{\infty})$ corresponding to multiplications on $E_{\infty}$.
\end{remark}

In order to simplify various curvature computations, it is useful to introduce the function
\begin{equation}\label{param}
F(z):= \Theta(z)\Mpc(z),
\end{equation}
where
$\Mpc(z) = \prod_{a \in \hat{\cA}} (1 + x_{a} z)^{d_{a}}$ and $d_a = {\rm dim}_{\C}(S_a)$ with $S_0 :=\C P^{d_0}$ and $S_{\infty}=\C P^{d_{\infty}}$.
The conditions \Ref{positivity}  imply
\begin{align}
\label{positivityF}
(i)\ F(z) > 0, \quad -1 < z <1,\quad
(ii)\ F(\pm 1) = 0,\quad
(iii)\ F'(\pm 1) = \mp 2p_c(\pm1).
\end{align}
When $\cA=\hat \cA$, \eqref{positivityF} is equivalent to \eqref{positivity}.

\smallskip
Letting $s_a := \frac{\pm Scal(\pm g_a)}{2d_a}$ be the normalized scalar curvatures of $S_a$ for $a\in \cA$ and $ s_{0} := d_{0} + 1, s_{\infty} :=
-(d_{\infty} +1)$, we recall the  following facts (see e.g. \cite{haml}).
\begin{lemma}\label{laplace}
For any admissible metric $g$, if $S(z)$ is a smooth function of z,
then
\begin{equation}
\label{Lapl}
\Delta_g S = -[F(z) S'(z)]'/p_c(z),\end{equation}
where $\Delta_g$ is the Laplacian of $g$, whereas the scalar curvature of  $g$ is given by
\begin{equation}
Scal(g) =  \sum_{a \in \hat{\cA}}
\frac{2 d_a s_a x_a }{1+x_az} - \frac{F''(z)}{\Mpc(z)}.
\label{scalarcurvature}
\end{equation}
\end{lemma}

\subsection{Admissible $(z+\bb, \p)$-extremal K\"ahler  metrics}
We now make the assumption that $g$ in \eqref{scalh1} is admissible and that the positive Killing potential $f$  is of the form $f=|z+\bb|$ for some constant $\bb \in {\mathbb R}$ such that $|\bb|>1$. In view of Remark~\ref{r:symmetry}, we shall assume (without loss of generality) that $\bb>1$.  It follows from \eqref{scalarcurvature} and Lemma~\ref{laplace} that the $(z+\bb, \p)$-scalar curvature is given by
\begin{equation}\label{scalh2}
\begin{array}{ccl}
Scal_{z+\bb,\p}(g) & = & \frac{-(z+\bb)^2 F''(z) + 2(\p-1)(z+\bb) F'(z) - \p(\p-1) F(z)}{\Mpc(z)} \\
\\
&+ &(z+\bb)^2 \sum_{a \in \hat{\cA}}
\frac{2 d_a s_a x_a }{1+x_az}.
\end{array}
\end{equation}
Thus, $g$  is  $(z+\bb,\p)$-extremal  in the sense of \cite{am,aclg17}, provided that
\begin{equation}\label{em1}
\begin{array}{cl}
& -(z+\bb)^2 F''(z) + 2(\p-1)(z+\bb) F'(z) - \p(\p-1) F(z)\\
\\
 =&  (A_1 z+A_2)\Mpc(z) - \Mpc(z)  (z+\bb)^2 \sum_{a \in \hat{\cA}}
\frac{2 d_a s_a x_a }{1+x_az}.
\end{array}
\end{equation}
Notice that $A_1=0$  in  \eqref{em1} is equivalent to $g$  having a constant $(z+\bb, \p)$ scalar curvature.

Next we will view the equation \eqref{em1} in an alternative way that will enable us to ensure the existence of a unique solution satisfying the boundary conditions (ii) and (iii) of \eqref{positivityF}. On closer inspection of the left hand side of \eqref{em1} we see that it equals $$-(z+\bb)^{\p+1}\frac{d^2}{dz^2}\left[ \frac{F(z)}{(z+\bb)^{\p-1}}\right]$$ and hence
\eqref{em1} is equivalent to
\begin{equation}\label{emNew}
\frac{d^2}{dz^2}\left[ \frac{F(z)}{(z+\bb)^{\p-1}}\right]=\frac{ \Mpc(z)}{(z+\bb)^{\p-1}} \sum_{a \in \hat{\cA}}\frac{2 d_a s_a x_a }{1+x_az}-\frac{ (A_1 z+A_2)\Mpc(z)}{(z+\bb)^{\p+1}}
\end{equation}
Letting
\begin{equation}\label{FG}
G(z) : =  \frac{F(z)}{(z+\bb)^{\p-1}} = \frac{\Theta(z)\Mpc(z)}{(z+\bb)^{\p-1}}
\end{equation}
and \begin{equation}\label{fbp}
Q(z) = \frac{ \Mpc(z)}{(z+\bb)^{\p-1}} \Big(\sum_{a \in \hat{\cA}}\frac{2 d_a s_a x_a }{1+x_az}\Big)-\frac{ (A_1 z+A_2)\Mpc(z)}{(z+\bb)^{\p+1}},
\end{equation}
we obtain the ODE
\begin{equation}\label{G-e}
G''(z) = Q(z)
\end{equation}
with boundary conditions
\begin{align}
\label{endG}
(i)\ G(\pm 1) = 0,\quad
(ii)\ G'(\pm 1) = \mp \frac{2p_c(\pm1)}{(\bb\pm 1)^{\p-1}}.
\end{align}
It is not hard to see that \eqref{G-e}-\eqref{endG} has a solution  if and only if
\begin{equation}\label{compatibility}
\begin{split}
& \frac{4p_c(-1)}{(\bb- 1)^{\p-1}}  + \int_{-1}^{1} Q(t) (1-t)\,dt=0,\\
&  \frac{2p_c(-1)}{(\bb- 1)^{\p-1}}+ \int_{-1}^{1} Q(t) \,dt =  -\frac{2p_c(1)}{(\bb+ 1)^{\p-1}},
\end{split}
\end{equation}
in which case the solution is
\begin{equation}\label{G}
G(z)= \frac{2p_c(-1)}{(\bb- 1)^{\p-1}} (z+1) + \int_{-1}^{z} Q(t) (z-t)\,dt.
\end{equation}
It is now just a matter of technical detailing to see that the necessary and sufficient conditions \eqref{compatibility} for the existence of the solution $G(x)$  above in fact  determine the constants $A_1$ and $A_2$, via the  linear system
\begin{equation}\label{A1andA2}
\begin{array}{ccc}
\alpha_{1, -(1+\p)} A_1 + \alpha_{0, -(1+\p)} A_2 =  2 \beta_{0, (1-\p)}\\
\\
\alpha_{2, -(1+\p)} A_1 + \alpha_{1, -(1+\p)} A_2 =  2\beta_{1, (1-\p)},
\end{array}
\end{equation}
where
\begin{equation}\label{alpha-beta-r-q}
\begin{split}
\alpha_{r,q} =& \int_{-1}^1 (t+ \bb)^q t^r p_c(t)dt \\
\beta_{r,q}   = &  \int_{-1}^1\Big(\sum_{a\in \hat{\cA}} \frac{x_ad_as_a}{1+x_at}\Big)  t^r p_c(t) (t+ \bb)^q dt \\
                        & + \big((-1)^r(\bb-1)^qp_c(-1) + (1+\bb)^q p_c(1)\big).
                         \end{split}
                         \end{equation}
Since $\alpha_{1,q}^2 < \alpha_{0,q}\alpha_{2,q}$,  \eqref{A1andA2} has a unique solution.  It  follows that the functions $Q(z)$  given by \eqref{fbp} and $G_{x,\bb, \p}(z):=G(z)$ with $G$ given by \eqref{G} are entirely determined from the data $(x_a, a \in \cA, \bb, \p)$, and thus are invariants of the admissible K\"ahler class $\Omega=\Omega_x$ on $M$.

We now note that if $F(z)$ satisfies \eqref{emNew}, or equivalently, $G(z)=F(z)/(z+\bb)^{\p-1}$ satisfies \eqref{G}, then $G''(z) = p_c'(z) \Upsilon (z)$ with $\Upsilon(-1)=2(d_0+1)/(\bb-1)^{\p-1}$, if $d_0>0$, and $\Upsilon(1) = -2(d_\infty +1)/(\bb+1)^{\p-1}$, if $d_\infty >0$.
Hence, assuming \eqref{endG}, we have $G'(z)=p_c(z) \Psi(z)$ with $\Psi(-1) = 2(d_0+1)/(\bb-1)^{\p-1}$ and $\Psi(1) = -2(d_\infty +1)/(\bb+1)^{\p-1}$, and $\Theta(\pm 1)=0$. Now by l'H\^opital's rule, $\Theta'(\pm 1) = \mp 2$, showing that  if the function $F_{x,\bb, \p}(x):= (z+\bb)^{\p-1}G_{x,\bb, \p}(z)$ satisfies  the positivity condition \eqref{positivityF} (i), then $\Theta(z)=F_{x,\bb, \p}(z)/p_c(z)$ gives rise to an admissible $(z+\bb, \p)$-extremal K\"ahler metric. We summarize this construction in the following
\begin{prop}\label{emsolution}  Let $M=P(E_0\oplus E_{\infty})\to S$ be an admissible manifold, $\Omega=\Omega_x$ an admissible  K\"ahler class corresponding to the admissible data $(x_a, g_a), a\in \cA$, and $\bb>1$ and $\p$ two real parameters. We denote by  $A_1,A_2$ the unique solution of \eqref{A1andA2} and let $F_{\Omega, \bb, \p}(z)= F_{x,\bb, \p}(z)=(z+\bb)^{\p-1}G_{x,\bb, \p}(z)$, with $G_{x, \bb, \p}(z)$ given by \eqref{G} for $Q(z)$ given by \eqref{fbp}, be the smooth function satisfying \eqref{emNew} and \eqref{positivityF} (ii)-(iii).  If $F_{\Omega,\bb, \p}(z)$  satisfies  (i) of \eqref{positivityF},  then we have a corresponding admissible K\"ahler metric $g$ on $M$ which is $(z+\bb, \p)$-extremal with respect to  the Killing potential $f(z)=z+\bb$. Furthermore, $g$ is $(z+\bb,\p)$-CSCK if, moreover,  $A_1=0$.
\end{prop}

\subsection{The $(K, \bb, \p)$-Mabuchi energy and the $(K, \bb, \p)$-Futaki invariant of an admissible manifold}\label{s:mabuchi-futaki}  We start by recalling the general setting of \cite{am,lahdili}. Let $(M, J)$ be a connected, compact $2m$-dimensional K\"ahler manifold  endowed  with a K\"ahler class $\Omega$. Suppose $K$ denotes a real holomorphic vector field with zeroes which is quasi-periodic,
i.e. whose flow generates a (real) torus. We fix a positive constant $\be>0$  and, for any $K$-invariant K\"ahler metric $\omega\in \Omega$,   let   $f_{\omega, K, \be}$ be the Killing potential of $K$ with respect to $\omega$, normalized by $\int_M f_{\omega,K,\be} \omega^m/m!=\be$. It is not hard to see that with this normalization, the image $f_{\omega,K,\be}(M)$ is an interval independent of  the choice of $\omega$ in $\Omega$.  We further require that $\be$ is chosen so that $f_{\omega,K,\be}>0$.

Let $\T$ be a maximal torus in the reduced automorphism group ${\rm Aut}^r(M, J)$ of $(M,J)$ with  $K\in {\rm Lie}(\T)$, and $\mathcal{K}^{\T}(M, J, \Omega)$ denote the space of $\T$-invariant K\"ahler metrics in $\Omega$, viewed as an affine space modelled on the vector space $C^{\infty}(M, \R)^{\T}/\R$ of $\T$-invariant smooth functions modulo constants. At each point $\omega\in \mathcal{K}^{\T}(M, J, \Omega)$, we identify the corresponding tangent space
$$T_{\omega} \mathcal{K}^{\T}(M, J, \Omega) \cong \Big\{ \dot \varphi \in C^{\infty}(M, \R)^{\T}  \ \big\vert \ \int_M \dot \varphi f_{\omega,K,\be}^{-(\p+1)} \frac{\omega^m}{m!} =0\Big\}.$$
Furthermore, we denote by $P_g^{\T}(M, \R)$ the finite dimensional space of Killing potentials with respect to $g=-J\omega$ of the vector fields in ${\rm Lie}(\T)$, and for a smooth function $\varphi  \in C^{\infty}(M, \R)^{\T},$ we let $\Pi^{g, K, b, \p,\T}(\varphi)$ denote its orthogonal projection to $P_g^{\T}(M, \R)$ by using the inner product
\begin{equation}\label{inner-product-0}
\langle \varphi, \psi \rangle_{\omega, K, \be, \p} = \int_M \varphi \psi f_{\omega,K,\be}^{-(\p+1)} \frac{\omega^m}{m!}
\end{equation}
on $C^{\infty}(M, \R)^{\T}$. We shall use the following definition from \cite{lahdili}.
 \begin{defn}\label{d:mabuchi}  The {\it relative $(K, \be, \p)$-Mabuchi energy}   is  a functional
 $${\mathcal M}^{\T}_{(\Omega, K, \be, \p)} : \mathcal{K}^{\T}(M, J, \Omega) \to \R,$$ defined,  up to an additive constant, by the property
 $$\big({\boldsymbol {\rm d}}{\mathcal M}^{\T}_{(\Omega, K, \be, \p)}\big)_{\omega}(\dot \varphi)= - \int_M (Scal_{f, \p}(g))^{{\perp}_g} \dot \varphi f^{-(\p+1)} \frac{\omega^m}{m!}, $$
where $f=f_{\omega, K,\be}$ is the Killing potential of $K$ with respect to $\omega$ and $Scal_{f, \p}(g)^{{\perp}_g} := Scal_{f, \p}(g) - \Pi^{g, K, \be, \p, \T}(Scal_{f, \p}(g))$  is {\it the reduced $(f,\p)$-scalar curvature}. The Killing  vector field $Z= J {\rm grad}_g\Big(\Pi^{g, K, \be, \p, \T}(Scal_{f, \p}(g))\Big)$ is independent of $\omega \in \mathcal{K}^{\T}(M, J,\Omega)$ and is called {\it the $(K, \be, \p)$-extremal vector field} associated to $(M, J, \Omega, \T)$. It vanishes if and only if the  $(K, \be, \p)$-{\it Futaki invariant} $\mathfrak{F}_{(\Omega, K, \be, \p)}: {\rm Lie}(\T) \to \R$  defined in \cite{am} and \cite[Def.~4]{lahdili} is zero.
\end{defn}
\begin{remark} In the case when  the $(K, \be, \p)$-extremal vector field of $(M, J, \Omega, \T)$ vanishes, one can also express the differential of ${\mathcal M}^{\T}_{(\Omega, K, \be, \p)}$ as
\begin{equation*}
\big({\boldsymbol {\rm d}}{\mathcal M}^{\T}_{(\Omega, K, \be, \p)}\big)_{\omega}(\dot \varphi)= - \int_M \big({Scal}_{f, \p}(g)-c_{(\Omega,K, b, \p)}\big)\dot \varphi f^{-(\p+1)} \frac{\omega^m}{m!},
\end{equation*}
where the constant
$$c_{(\Omega, K, \be,  \p)} = \frac{\int_M {Scal}_{f, \p}(g)f^{-(\p+1)} \frac{\omega^m}{m!}}{\int_M f^{-(\p+1)} \frac{\omega^m}{m!}}$$
is independent of $\omega \in \mathcal{K}^{\T}(M, J,\Omega)$, see \cite{am,lahdili}. Thus, in this case,  ${\mathcal M}^{\T}_{(\Omega, K, \be, \p)}$ reduces to the  Mabuchi functional introduced in \cite{lahdili, lahdili2}.
\end{remark}

We now specialize to the case when $M$ is an admissible manifold with CSCK base, and $K$ is the generator of the natural $\Sph^1$-action.
It is shown in \cite[Prop. 5]{acgt}  that  an admissible K\"ahler metric $(g, \omega, J)$  is invariant under a common maximal compact connected subgroup $G \subset {\rm Aut}^r(M,J)$ with $K\in {\rm Lie}(G)$.  We thus  fix a maximal torus $\T\subset G$ with $K \in {\rm Lie}(\T)$. Notice that a Killing potential of $K$ with respect to $g$ is  given by $z+\bb$ for some $\bb \in \R$. The constant $\bb$ here is essentially the real constant $\be$ in the above general setting: indeed, $\bb$ and $\be$ are linked by an affine-linear expression which is independent of the choice of $g \in \mathcal{K}^{\rm adm}(M, \omega)$. For this reason, in the admissible context, we shall use $\bb$ instead of $\be$, thus referring to  the relative $(K, \bb, \p)$-Mabuchi energy and $(K, \bb, \p)$-Futaki invariant for the corresponding quantities defined for  a fixed $\bb>1$ (and a maximal torus $\T$ as above).

We shall first compute  the reduced  scalar curvature $Scal_{z+\bb, \p}(g)^{{\perp}_g}$, as  defined in Definition~\ref{mabuchi} (similarly to \cite[Prop. 6]{acgt}). The formula \eqref{scalh2}  reads as
\begin{equation}\label{scalh3}
Scal_{z+\bb, \p}(g)= (z+ \bb)^2\Big(\sum_{a} \frac{2d_a s_a x_a}{1+ x_az} \Big)- \frac{(z+\bb)^{\p+1}}{p_c(z)}\Big(\frac{F(z)}{(z+\bb)^{\p-1}}\Big)^{''}.
\end{equation}
For $A_1, A_2$ given by \eqref{A1andA2}, integration by parts of \eqref{scalh3} shows that $Scal_{z+\bb, \p}(g)- A_1z - A_2$ is $L^2$-orthogonal to $1$ and $z$ with respect to the measure $p_c(z)(z+ \bb)^{-(\p+1)} dz$. Geometrically, this means that $Scal_{z+\bb, \p}(g)- A_1z - A_2$ is $\langle \cdot, \cdot \rangle_{\omega, K, \bb,  \p}$-orthogonal to the Killing potentials $z$ and $1$, where $\langle \cdot, \cdot \rangle_{\omega, K, \bb,  \p}$ stands for the inner product \eqref{inner-product-0} corresponding to the Killing potential $f_{\omega,K, b}= z+ \bb$. By \cite[Prop.~2]{acgt}, any other Killing potential for a vector field in ${\rm Lie}(\T)$   has the form $\sum_{a}(1+ x_az)f_a$ where $f_a$ is a Killing potential  for the base factor $S_a$, which we can assume without loss is of zero mean with respect to $(\hat S, g_{\hat S})$. Formulae \eqref{scalh3} then shows that  $Scal_{z+\bb, \p}(g)- A_1z - A_2$ will be $\langle \cdot, \cdot \rangle_{\omega, K, \bb,  \p}$-orthogonal to such Killing potentials.  In particular,  $A_1K$ is the $(\Omega, K, \bb, \p)$-extremal vector field  of $(M, J, \Omega, K, \bb),$ computed with respect to a maximal torus $\T\subset {\rm Aut}^{r}(M,J)$ (see \cite{am} and \cite[Def.~7]{lahdili}).

Using that $F_{\Omega,\bb,\p}(z)$ is a solution of  \eqref{emNew}, we have
\begin{lemma}\label{reduced-scalar} Let $(M, J, g, \omega)$  be an admissible K\"ahler manifold over a  CSCK base and $\T$ a maximal torus in the isometry group of $(g, \omega)$. If $g$ is parameterized
by the function $F(z)$ given in \eqref{param}, then the reduced $(z+\bb, \p)$-scalar curvature $Scal_{z+\bb, \p}(g)^{{\perp}_g}$  is given by
\begin{equation}\label{scalh4}
\begin{split}
Scal_{z+\bb, \p}(g)^{{\perp}_g} &= -A_1z - A_2 +(z+ \bb)^2\Big(\sum_{a} \frac{2d_a s_a x_a}{1+ x_az}\Big) - \frac{(z+\bb)^{\p+1}}{p_c(z)}\Big(\frac{F(z)}{(z+\bb)^{\p-1}}\Big)^{''} \\
                                          &=\frac{(z+\bb)^{\p+1}}{p_c(z)}\Big(\frac{F_{\Omega,\bb,\p}(z)}{(z+\bb)^{\p-1}} - \frac{F(z)}{(z+\bb)^{\p-1}}\Big)'',
                                          \end{split}
\end{equation}
where $F_{\Omega,\bb,\p}(z)$ is the smooth function defined in Proposition~\ref{emsolution} in terms of $(\Omega, \bb, \p)$. Furthermore, the $(K, \bb, \p)$-extremal vector field  is $A_1K$.
\end{lemma}

A direct corollary is the following
\begin{prop}\label{futaki} Let $(M, J, g, \omega)$  be an admissible K\"ahler manifold over a  CSCK base,  and $\T$ a maximal torus in the isometry group of $(g, \omega)$.  Then,  the corresponding  $(K, \bb, \p)$-Futaki invariant $\mathfrak{F}_{([\omega], K, \bb, \p)}$ restricted to ${\rm Lie}(\T)$ vanishes  iff $A_1=0$. The latter condition is equivalent to $\mathfrak{F}_{[\omega],K, \bb, \p}(K)=0$.
\end{prop}

We now give an explicit form for the relative $(K,\bb, \p)$-Mabuchi energy in the admissible case, following the similar construction in \cite[Prop. 7]{acgt}.  To this end, we use the parametrization of $\mathcal{K}^{\rm adm}(M, \omega)$ in terms of fibre-wise symplectic potentials $u(z)$ (defined by $u''(z)=\frac{\Mpc(z)}{F(z)}$, see Sect.~\ref{s:admissible-metrics}) and \eqref{symplectic/complex}.
\begin{prop}\label{mabuchi} Let $(M, J, g, \omega)$  be an admissible K\"ahler manifold over a  CSCK base,  and $\T$ a maximal torus in the isometry group of $(g, \omega)$. Then,  the relative $(K,\bb, \p)$-Mabuchi energy associated to $\Omega=[\omega]$ and $\T$,  restricted to the space of admissible K\"ahler metrics $\mathcal{K}^{\rm adm}(M, \omega)$ is given (up to an additive constant) by a positive multiple of the functional
\begin{equation*}
\begin{split}
\mathcal{M}_{g_c} : u(z) \longmapsto & \int_{-1}^1 \frac{F_{\Omega,\bb,\p}(z)}{(z+\bb)^{p-1}}\big(u''(z)-u''_c(z)\big)dz \\
                                                             &- \int_{-1}^{1}\frac{p_c(z)}{(z+\bb)^{p-1}}\log\Big(\frac{u''(z)}{u''_c(z)}\Big)dz,
                                                      \end{split}
                                                      \end{equation*}
                                                      where $F_{\Omega,\bb,\p}(z)$ is the smooth function defined in Proposition~\ref{emsolution}, and $u_c(z)$ is the fibre-wise symplectic potential for some fixed $\omega$-compatible admissible K\"ahler metric $g_c\in \mathcal{K}^{\rm adm}(M, \omega)$.
                                                      \end{prop}
\begin{proof} Using Lemma~\ref{reduced-scalar}, the proof is identical to the one of  \cite[Prop.~7]{acgt}. \end{proof}
The proof of \cite[Cor.~3]{acgt} yields
\begin{prop}\label{Lahdili} {\it Let $(M, J, g, \omega)$  be an admissible K\"ahler manifold over a  CSCK base,  and $\T$ a maximal torus in the isometry group of $(g, \omega)$.  If the function $F_{[\omega],\bb,\p}(z)$ is strictly negative somewhere on $(-1,1)$, then the {relative $(K,\bb, \p)$-Mabuchi energy}  of $(M, J, [\omega], \T)$ is unbounded from below. }
\end{prop}

\subsection{The $(\hat \beta_{\bb}, \p)$-Donaldson--Futaki invariant}\label{beta-futaki} In \cite[Sect.~5.1]{am}, a quantized version of the $(f, 2m)$-Donaldson--Futaki invariant was proposed, which leads to a notion of $(\beta, 2m)$-K-stability, where $\hat \beta$ is a fixed $\C^*$-subgroup in the automorphism group ${\rm Aut}(M, L)$ of a smooth  compact polarized variety $(M, L)$.  This was further developed and generalized in \cite{lahdili2} for arbitrary weights $\p$  and for quasi-periodic vector fields $\hat K \in {\rm Lie}({\rm Aut}(M, L))$.  We first briefly recall the  general setting of \cite{am,lahdili2}.

\smallskip
Let $(M, L)$ be a smooth compact polarized projective variety and $\Omega = 2\pi c_1(L)$ the corresponding K\"ahler class. Let $\hat \beta$ be a $\C^*$-subgroup of ${\rm Aut}(M,L)$,   which covers a $\C^*$-subgroup  $\beta$  of the reduced automorphisms group ${\rm Aut}^{r}(M, J) \cong {\rm Aut}(M,L)/ \{{\mathbb C}^{*}\cdot{\rm Id}_{L}\}$. We denote by $\hat K_{\beta}$ (resp. $K_{\beta}$) the generator
of the corresponding $\Sph^1$-action on $L$ (resp. on $M$),  and by $B_k$ the infinitesimal generators  for  the induced linear  $\C^*$-actions  on $H^0(M, L^k), k\ge 1$. We use the following normalization for $B_k$: for any holomorphic section $s\in H^0(M, L^k)$ and $x \in M$,
\begin{equation}\label{normalization}
(B_k \cdot s) (x) : = i \frac{d}{dt}_{|t=0} \Big( \varphi_t^{\hat K_{\beta}}\big(s(\varphi_{-t}^{K_{\beta}}(x))\big)\Big),
\end{equation}
where  $\varphi_t^{\hat K_{\beta}}$ and $\varphi_t^{K_{\beta}}$ are the flows of $\hat K_{\beta}$  and $K_{\beta}$, respectively. Notice that $B_k$ is hermitian (and therefore semi-simple with real eignvalues) with respect to the $L^2$ inner product defined by any $\hat K_{\beta}$-invariant hermitian metric on $L$.  We further assume that the $\C^*$-action $\hat \beta$ on $L$ is such that $B_k$  have positive eigenvalues for any $k$ large enough: this is equivalent with the property  that for any $\hat K_{\beta}$-invariant hermitian product $h$ on $L$, the induced momentum map  $f_{\beta}$  for $K_{\beta}$ with respect to  the curvature form $\omega_h \in \Omega$  is positive. Indeed, if $\nabla^{h,k}$ is the corresponding Chern connection on $L^k$,  it is  well-known (see e.g.~\cite{gauduchon-book}) that
\begin{equation}\label{generator}
B_k =-i \nabla^{h,k}_{K_{\beta} }+  k f_{\beta}.
\end{equation}
It then follows that each eigenvalue $\lambda$ of $B_k/k$ equals $f_{\beta}(x_0)$  for a point $x_0$ of maxima of $|s|_{h,k}^2$  of an eigensection $s$ corresponding to $\lambda$.

We now consider a  {\it normal} $\hat \beta$-compatible test configuration (of exponent $r$) $({\mathcal X}, \cL)$ associated to $(M, L)$. By this, we mean that $({\mathcal X}, \cL)$ is a normal polarized variety of complex dimension $(m+1)$,  endowed with a $\C^*$-equivariant map $\pi : {\mathcal X} \to \C,$ such that
\begin{enumerate}
\item[$\bullet$] $(M_t= \pi^{-1}(t), L_t=\cL_{|_{M_t}}) \cong (M, L^r)$ for $t\neq 0$, and
\item[$\bullet$] $\pi$ is flat (and therefore the  dimensions of $H^0(M_t, L_t^k)$ stay unchanged for $t\in \C$ and $k$ large enough),
\item[$\bullet$] $\pi$ is $\hat \beta$-invariant.
\end{enumerate}
It follows that the central fibre $(M_0, L_0)$ of such a test configuration is an (in general singular) polarized projective variety,  endowed with two commuting $\C^*$-actions, $\hat \alpha$ and $\hat \beta$. We denote respectively by $A_k, B_k$ the corresponding infinitesimal generators for the induced linear $\C^*$-actions on $H^0(M_0, L_0^k)$. For each $q \in \R$, we expect to have expansions
\begin{equation}\label{asymptotics}
\begin{split}
k^{-m +1-q} {\rm Tr}( B_k^q) &= k b_0^{q,0}(\hat \beta)  + b_1^{q,0}(\hat \beta) + O(k^{-1}), \\
k^{-m-q}{\rm Tr} (A_k B_k^q) &= k b_0^{q, 1}(\hat \beta, \hat \alpha) + b_1^{q, 1}(\hat \beta, \hat \alpha) + O(k^{-1}).
\end{split}
\end{equation}
To the best of our knowledge, such  Hilbert expansions hold for $q \in \N$ (see \cite{donaldson06}) and, for any $q$,   if we assume that $M_0$ is smooth (this follows from the considerations in \cite{lahdili2,lahdili3}) or that $M_0$ is a toric variety \cite{am}. We shall exhibit below another situation where the expansions \eqref{asymptotics} hold.

\smallskip
Assuming that \eqref{asymptotics} do hold on $(M_0, L_0)$, we define the $(\hat \beta, \p)$-Donaldson--Futaki invariant  of the test configuration $({\mathcal X}, \cL)$ as follows.

\begin{Definition}\label{K-stability} The $(\hat \beta, \p)$-Donaldson--Futaki invariant of a normal $\hat \beta$-compatible test configuration $({\mathcal X}, \cL)$  associated to $(M, L)$ such that \eqref{asymptotics} holds true on the central fibre $(M_0, L_0)$  is defined to be
\begin{equation}\label{Donaldson-Futaki}
{\rm DF}_{(\hat \beta, \p)}({\mathcal X}, \cL) := \frac{b_1^{(1-\p), 1}(\hat\beta, \hat \alpha)b_0^{-(\p+1),0}(\hat\beta)-b_0^{-(\p+1), 1}(\hat\beta, \hat \alpha) b_1^{(1-\p),0}(\hat \beta)}{b_0^{-(\p+1),0}(\hat\beta)}.
\end{equation}
The polarized variety $(M, L)$ is called {\it $(\hat\beta, \p)$-K-semistable} if for any normal, $\hat\beta$-compatible test configuration $({\mathcal X}, \cL)$ for $(M,L)$ as above,
\begin{equation}\label{semistable}
{\rm DF}_{(\hat\beta, \p)}(\mathcal X, \cL) \ge 0.
\end{equation}
$(M, L)$ is {\it $(\hat \beta, \p)$-K-stable} if, furthermore, equality in \eqref{semistable} holds if and only if ${\mathcal X}= M \times \C, \cL= L\otimes \cO_{\C}$ is a product test configuration (see Remark~\ref{quasi-periodic} part (2) below).
\end{Definition}
\begin{remark}\label{quasi-periodic} (1) The definition of $(\hat\beta, \p)$-K-stability makes sense for any positive multiple $\lambda$ of $\hat \beta$, i.e. taking $\lambda B_k$ instead of $B_k$. This will introduce an overall positive factor in the computation of ${\rm DF}_{(\lambda \hat \beta, \p)}$. More generally, one can define $(\hat K, \p)$-K-stability with respect to a quasi-periodic vector field $\hat K \in {\rm Lie}({\rm Aut}(M, L))$, by  considering the linear operators $B_k$ acting on $H^0(M_0,  L_0)$ via \eqref{generator} for a $\hat K$-invariant hermitian product $h$ on $L$, see \cite{lahdili2}.

(2) In the special case of the product test configuration  ${\mathcal X} = M \times \C, \cL = L \otimes \cO_{\C},$ the expression on the rhs of \eqref{Donaldson-Futaki} can be regarded as defining a numerical invariant  ${\rm DF}_{(\hat\beta, \p)}(\hat \alpha)$  associated to any $\C^*$-subgroup  $\hat \alpha \subset {\rm Aut}(M,L)$.  It is  possible to see (this follows essentially from \cite{lahdili2}) that the expansions \eqref{asymptotics} hold and  ${\rm DF}_{(\hat\beta, \p)}(\hat \alpha)$ coincides, up to a positive multiplicative constant, with the differential-geometric $(f_{\hat \beta}, \p)$-Futaki invariant $\mathfrak{F}_{\Omega, f_{\hat \beta}, \p}$ evaluated on the $\Sph^1$-generator of $\alpha$. (Recall that $f_{\hat \beta}$ is the Killing potential of $\beta$ determined from $\hat \beta$, see \eqref{generator}.) In particular, ${\rm DF}_{(\hat\beta, \p)}(\hat \alpha)$ does not depend on the lift $\hat \alpha$ of the   $\C^*$-action $\alpha \subset {\rm Aut}^r(M,J)$. At times we will even replace $\alpha$ in this notation by the corresponding $\Sph^1$-generator.
\end{remark}

Following \cite{sz}, one can also define a {\it relative} version of the $(\hat \beta,\p)$-Donaldson--Futaki invariant. To this end, we suppose that $(M, L)$ is a compact smooth polarized variety, $\T \subset {\rm Aut}^{r}(M,J)$ a fixed maximal torus,  $\beta \subset \T^c$  with a lift $\hat \beta \subset {\rm Aut}(M, L)$ as before. Here $\T^c$ denotes the complexification of $\T$. For any two  $\C^*$-actions $\gamma', \gamma'' \subset \T^c$, we suppose that there exist an expansion
\begin{equation}\label{general-expansion}
k^{-m-q-2}{\rm Tr} (C'_k C''_k B_k^q) = c_0^{q, 0}(\hat \beta, \hat \gamma', \hat \gamma'') + O(k^{-1}),
\end{equation}
where $C_k'$ and $C_k''$ are the corresponding generators on the space $H^0(M, L^k)$ for some lifts $\hat \gamma'$ and $\hat \gamma''$ of $\gamma'$ and $\gamma''$ to ${\rm Aut}(M,  L)$. We then define the $(\hat \beta, \p)$-weighted product for the $\C^*$-actions $\gamma', \gamma''$ by
\begin{equation}\label{inner-product}
\langle \gamma', \gamma'' \rangle_{(\hat\beta, \p)} := c_0^{-(1+\p),0}(\hat \beta, \hat \gamma', \hat \gamma'') -b_0^{-(1+\p), 0}(\hat \beta, \hat \gamma')b_0^{-(1+\p),0}(\hat \beta, \hat \gamma'').
\end{equation}
The above definition does not depend on the lifts of $\gamma'$ and $\gamma''$ to ${\rm Aut}(M, L)$: it computes the zero order coefficient of the expansion
$$k^{-m+ \p-1}{\rm Tr} ({\mathring C}'_k {\mathring C}''_k B_k^{-(1+\p)}) = \langle \gamma', \gamma'' \rangle_{(\hat\beta, \p)} + O(k^{-1}),$$
where ${\mathring C}_k= C_k - c_k k {\rm Id}$  for a constant $c_k$, uniquely  determined by  the condition (see  \eqref{asymptotics}) $$k^{-m+ \p}{\rm Tr} ({\mathring C}_k B_k^{-(1+\p)}) =   O(k^{-1}).$$ Using the fact that the eigenvalues of $B_k/k$ are  in the interval $[(f_{\beta})_{\rm min}, (f_{\beta})_{\rm max}] \subset (0, \infty)$, one sees that \eqref{inner-product} defines (by linearity) an inner product on $\mathfrak{t} = {\rm Lie}(\T)$.

We now consider a normal  test configuration  $({\mathcal X}, \cL)$ associated to a smooth polarized variety $(M,L)$ as before, and also assume that it is compatible with a fixed maximal torus $\T \subset {\rm Aut}^{r}(M,J)$ (with  $\beta \subset \T^c$ as before). Thus,  $\T$ is subtorus of a maximal torus $\T_0 \subset {\rm Aut}(M_0, L_0)/ \{\C^*\cdot {\rm Id}_{L_0}\}$  of the reduced autmorphism group of  the central fibre $(M_0, L_0)$, and we have an embedding $\mathfrak{t} \subset \mathfrak{t}_0$ of the corresponding Lie algebras.  We denote by $\gamma_{\rm ex}\in \mathfrak{t} \subset \mathfrak{t}_0$ the element corresponding to the $K_{(\hat \beta, \p)}$-extremal vector field of $(M, \Omega, \T)$, where $\Omega = 2\pi c_1(L)$. We also assume that the inner product $\langle \cdot, \cdot \rangle_{(\hat\beta, \p)}$ is well-defined on $\mathfrak{t}_0$, i.e.  the expansions \eqref{general-expansion} hold on
$(M_0, L_0)$.  We then define (by using Remark~\ref{quasi-periodic} part (2))
\begin{defn}\label{relative-K-stability} The relative $(\hat \beta, \p)$-Donaldson--Futaki invariant ${\rm DF}^{\gamma_{\rm ex}}_{(\hat\beta, \p)}(\mathcal X, \cL)$ of a test configuration $(\mathcal X, \cL)$ as above is defined by
$${\rm DF}^{\gamma_{\rm ex}}_{(\hat\beta, \p)}(\mathcal X, \cL):= {\rm DF}_{(\hat\beta, \p)}(\mathcal X, \cL) - \left[\frac{\langle \alpha, \gamma_{\rm ex} \rangle_{(\hat\beta, \p)}}{\langle \gamma_{\rm ex}, \gamma_{\rm ex} \rangle_{(\hat \beta, \p)}}\right]{\rm DF}_{(\hat \beta, \p)}(\gamma_{\rm ex}).$$
The polarized variety $(M, L)$ is said to be {\it relative $(\hat\beta, \p)$-K-semistable} if for each $\T$ compatible normal  test configuration $(\mathcal X, \cL)$ as above, ${\rm DF}^{\gamma_{\rm ex}}_{(\hat\beta, \p)}(\mathcal X, \cL) \ge 0$.  It is {\it relative $(\hat \beta, \p)$-K-stable} if, furthermore, equality in the latter inequality holds if and only if $(\mathcal X, \cL)$ is  a product configuration.
\end{defn}

\smallskip
We now explore Definitions~\ref{K-stability}  and \ref{relative-K-stability} in the case when $\Omega$  is an admissible K\"ahler class on $M=P(E_0 \oplus E_{\infty})$,   which is also a rational multiple of a polarization, i.e. $r \Omega = 2\pi c_1(L)$ for a holomorphic line bundle $L$ over $M$ and a positive integer $r$. As the theory is homogeneous in $r$, we shall assume without loss $r=1$ and refer to $L$ as an {\it admissible polarization}.

In \cite{acgt,sz} (following \cite{RT}) a $1$-parameter family of  test configurations associated to an admissible polarized variety $(M, L)$ is constructed as follows: let $\mathcal X$ be the {\it degeneration to the normal cone} of the divisor $e_{\infty} := P(0\oplus E_{\infty}) \subset M$.  Thus, $\pi : \mathcal X \to \C$ is a (smooth) polarized variety obtained by blowing up $M \times \C$  along $e_{\infty}\times \{0\}$. Notice that $M_0=\pi^{-1}(0)$ consists of two  (smooth) varieties:  the exceptional divisor $P$ in $\mathcal X$  and the  blow-up $\hat M$ of $M$ along  $e_{\infty}$,  intersecting at the exceptional divisor $E= P(\nu_{e_{\infty}})$ of $\hat M$, where $\nu_{e_{\infty}}:= TM_{|e_{\infty}}/Te_{\infty} \to e_{\infty}$ is the normal bundle of $e_{\infty}$.

We denote by $\hat \alpha$ the induced $\C^*$-action on $(M_0,L_0)$.  It is shown in \cite{RT} that the Seshadri constant of $e_{\infty}$ with respect to $L$  is $2$, which means that  there is  a $1$-parameter family of  polarizations $\cL_c=\pi^*(L)\otimes \cO(-cP), \ c\in (0, 2)\cap \Q$ of $\mathcal X$ with  $(M_t, (\cL_c)_{|M_t}) \cong (M, L)$, thus giving rise to the family $(\mathcal X, \cL_c)$ of test-configurations associated to $(M, L)$. Letting $\z:=c-1$ (this is a formal substitution),  we have, for $k$ sufficiently large,   the following $\hat \alpha$-invariant decomposition of $H^0(M_0, L_0^k)$ (see \cite{acgt,sz})
\begin{equation}\label{weight}
\begin{split}
H^0(M_0, L_0^k) = & \bigoplus_{i=0}^{(1-\z)k} H^0(e_{\infty}, L^k_{|e_{\infty}}\otimes S^{2k-i} \nu_{\infty}^*) \\
                       &  \bigoplus_{j=1}^{(1+\z)k} H^0(e_{\infty}, L^k_{|e_{\infty}}\otimes S^{(1+z)k-j} \nu_{\infty}^*) \\
                       =& \bigoplus_{i=0}^{2k}  H^0(e_{\infty}, L^k_{|e_{\infty}}\otimes S^{2k-i} \nu_{\infty}^*).
                       \end{split}
                       \end{equation}
It is shown in \cite{acgt,sz} that  \eqref{weight} gives rise to the eigenspace decomposition of  the generator $A_k$ for the action of $\hat \alpha$ on $H^0(M_0, L_0^k)$  as follows:  each summand of on the first line corresponds to the eigenvalue $0$ whereas (with the normalization \eqref{normalization} for $A_k$)  the summands on the second line correspond to eigenvalues $j$. In the above equalities, $\nu_{\infty}^*$ denotes the dual of the normal bundle of  $e_{\infty}$.

Let us now endow the admissible manifold $M=P(E_0 \oplus E_{\infty})$ with an admissible K\"ahler metric $(g,\omega)$ in the admissible K\"ahler class  $\Omega=2\pi c_1(L)$: we can take for instance  the canonical K\"ahler metric \eqref{g}. We denote by $K$ the vector field on $M$ generating the $\Sph^1$-action $\beta(e^{i\varphi})\cdot [a, b] = [e^{i\varphi}a, b]=[a, -e^{i\varphi}b]$.  Recall that $(g,\omega)$ is $K$-invariant by construction, and $z$ is the momentum map of $K$ with respect to $\omega$ whereas $e_{\infty}= z^{-1}(-1)$. We denote by $h$ the hermitian metric on $L$ whose curvature form is $\omega$ and use its Chern connection to lift $K$ to  a holomorphic vector field $\hat K = K^H - (z+ \bb) T$ on the total space of $L$, where $K^H$ is the horizontal lift and $T$ is the vector field generating the multiplications by $e^{i\varphi}$ on each fibre of $L$. It is well-known (see e.g. \cite{gauduchon-book}) that for suitable values of $\bb$, $\hat K$ generates an $\Sph^1$-subgroup in ${\rm Aut}(M,L)$. Restricting $\hat K$ to $L_{|e_{\infty}}$ (where $z=-1$ and $K^H=0$), we see that these values are $\bb \in \Z$. We denote by $\hat \beta_{\bb}$ the corresponding lift  of the $\Sph^1$-action $\beta$ to $L$, and by $B_{k, \bb}$ the generator for the action on $H^0(M_0, L^k_0)$. Notice that $\hat \beta_{\bb}$ acts fibre-wise with weight $(-\bb +1)$  on $L_{|e_{\infty}}$, and  with weight $1$ on the normal bundle $\nu_{\infty}$ of $e_{\infty}$: The latter follows for instance by computing the eigenvalues of the hessian of $z$ with respect to $g$ (or equivalently of $dJdz= d(1-z^2)\theta$ with respect to $\omega$) along $e_{\infty}=z^{-1}(-1)$,  by using the explicit form of the metric \eqref{g} and \eqref{theta}. Thus, under the normalization \eqref{normalization}, $B_{k,\bb}$  acts  on $H^0(e_{\infty}, L^k_{|e_{\infty}}\otimes S^{uk+v} \nu_{\infty}^*)$  as $(k(u+\bb-1)+v){\rm Id}$.

By Remark~\ref{quasi-periodic} part (1), in relation to questions of stability, we can consider
more generally the quasi-periodic vector fields $\hat K = K^H - (z+\bb)T$ with $\bb \in ] 1 , +\infty)$. With a small abuse of notation, we shall continue to refer to the corresponding $(\hat K, \p)$-Futaki--Donaldson invariant as ${\rm DF}_{(\hat \beta_{\bb}, \p)}$ and to  the $(\hat K, \p)$-K-stability notion as  $(\hat \beta_{\bb}, \p)$-K-stability. We now prove the following.
\begin{prop}\label{futaki-ruled} Let $M = P(E_0 \oplus E_{\infty})\to S$ be an admissible manifold over a CSCK base, and $L$ an admissible polarization of $M$,  which defines, up to a scale, an admissible K\"ahler class $\Omega$ with  $\mathfrak{F}_{\Omega, K, \bb, \p}(K)=0$, see Proposition~\ref{futaki}. Then, the $(\hat \beta_{\bb}, \p)$-Donaldson--Futaki  invariant of the test configuration  $(\mathcal X, \cL_{\z})$ with parameter $\z\in \Q \cap (-1, 1),$ corresponding to the degeneration of the normal cone of the infinity section $e_{\infty}=P(0 \oplus E_{\infty})$,  is given up to a positive scale by $F_{\Omega,\bb,\p}(\z)$.
\end{prop}
\begin{proof} The computations are essentially  identical to those in \cite[pp. 589--591 ]{acgt}, so we shall be brief. The dimension of $H^0(e_{\infty}, L_{|e_{\infty}}^k \otimes S^{uk+v}(\nu^*_{\infty}))$  is computed in \cite{acgt} to be (up to a common positive constant which we shall ignore)
\begin{equation}\label{multiplicity}
k^{m-1}p_c^s(u-1+ v/k) + O(k^{m-3}),
\end{equation}
where $p_c^s(t):=\prod_{a\in \cA}(1+x_a(t+ s_a/2k))^{d_a}$.

In order to compute ${\rm Tr}(B_{k,\bb})^q$, we shall use that for any smooth function on the interval $[0, R]$, $R\in \Q_{>0}$, we have
\begin{equation}\label{riemann-sum}
\sum_{i=\varepsilon}^{Rk} f\Big(\frac{i}{k}\Big) =k \int_{0}^R f(t)dt + \frac{1}{2}\Big(f(R) + (-1)^{\varepsilon} f(0)\Big) + O(k^{-1}),
\end{equation}
where $\varepsilon={0,1}$. The above estimate is established in \cite{RT} for  a polynomial (see also \cite[Lemma~9]{acgt}), but it  also holds for any smooth function
 $f$, see~\cite{GS,zeld}.

We have already shown that $(B_k/k)^q$ acts on $H^0(e_{\infty}, L_{|e_{\infty}}^k \otimes S^{2k-i}(\nu^*_{\infty}))$ as $((1+\bb) -i/k){\rm id}$. We thus obtain (by using \eqref{multiplicity} and \eqref{riemann-sum})
\begin{equation}\label{1}
\begin{split}
k^{-q -m +1}{\rm Tr}(B_{k,\bb})^q  = &\Big(\sum_{i=0}^{2k}\big((1+\bb) - i/k\big)^q p_c^s(1-i/k)\Big) + O(k^{-1}) \\
                             = & k \int_{0}^2 (1+\bb-t)^qp_c^s(1-t) dt \\
                                 & + \frac{1}{2}\big((\bb-1)^qp_c^s(-1) + (1+\bb)^qp_c^s(1)\big) + O(k^{-1}) \\
                             = & k \int_{-1}^{1}(t+ \bb)^q p_c(t)dt  \\
                                & + \frac{1}{2}  \int_{-1}^1\Big(\sum_{a\in \cA} \frac{x_ad_as_a}{1+x_at}\Big)  p_c(t) (t+ \bb)^q dt \\
                                & + \frac{1}{2}\big((\bb-1)^qp_c(-1) + (1+\bb)^qp_c(1)\big) + O(k^{-1}) \\
                               =&   k \alpha_{0,q} + \frac{1}{2} \beta_{0,q} + O(k^{-1}),
\end{split}
\end{equation}
where $\alpha_{r,q}$ and $\beta_{r,q}$  are defined by \eqref{alpha-beta-r-q}.

Similarly,  we have
\begin{equation}\label{2}
\begin{split}
k^{-q-m} {\rm Tr} (A_k B_{k,\bb}^q) =& \sum_{j=1}^{(1+\z)k} (j/k)\big((\z+\bb) - j/k\big)^q p_c^s(\z-j/k)\Big) + O(k^{-2}) \\
                                                   =& k \int_{0}^{1+\z}t\big((\z+\bb)-t\big)^qp_c^s(\z-t)dt \\
                                                       &+ \frac{1}{2}(1+\z)(\bb-1)^q p_c(-1)  + O(k^{-1})\\
                                                      =&  k \int_{-1}^{\z} (\z-t)(t+\bb)^qp_c(t) dt  \\
                                                         &+ \frac{1}{2}\int_{-1}^{\z} \Big(\sum_{a\in \cA}\frac{d_as_ax_a}{1+x_at}\Big)(\z-t)(t+\bb)^q p_c(t) dt  \\
                                                         &+\frac{1}{2}(1+\z)(\bb-1)^q p_c(-1)  + O(k^{-1}).
\end{split}
\end{equation}
Notice that \eqref{1}-\eqref{2} show that the expansions \eqref{asymptotics} hold for our test configurations (for any value of the parameter $\z\in (-1, 1)\cap \Q)$. Our remaining task is to compute the corresponding $(\hat \beta_{\bb}, \p)$-Donaldson--Futaki invariants.

\smallskip
By \eqref{1}, we obtain
\begin{equation}\label{c}
b_{0}^{-(\p+1),0}(\hat \beta_{\bb}) = \alpha_{0, -(\p+1)},\qquad b_1^{1-\p,0}(\hat \beta_{\bb}) = \frac{1}{2} \beta_{0, 1-\p}.
\end{equation}
whereas \eqref{2} yields
\begin{equation*}
\begin{split}
b_{0}^{-(\p+1),1} (\hat \beta_{\bb},\hat \alpha)= &\int_{-1}^{\z} (\z-t)(t+\bb)^{-(\p+1)}p_c(t) dt  \\
b_{1}^{(1-\p), 1}(\hat \beta_{\bb}, \hat \alpha)  =  & \frac{1}{2}\int_{-1}^{\z} \Big(\sum_{a\in \cA}\frac{d_as_ax_a}{1+x_at}\Big)(\z-t)(t+\bb)^{1-\p} p_c(t) dt  \\
                                                         &+\frac{1}{2}(1+\z)(\bb-1)^{1-\p} p_c(-1).
\end{split}
\end{equation*}
Substituting in \eqref{Donaldson-Futaki},  and using also $A_1=0$ (so that, by the first relation in \eqref{A1andA2} $2\beta_{0,-\p+1}= \alpha_{0,-(1+\p)}A_2$),  \eqref{fbp} and \eqref{G}, we get
\begin{equation*}
\begin{split}
{\rm DF}_{(\hat \beta_{\bb},\p)} (\mathcal X, \cL_{\z}) = & b_1^{(1-\p),1}(\hat \beta_{\bb}, \hat \alpha) -\left(\frac{b_1^{1-\p,0}(\hat \beta_{\bb})}{b_0^{-(1+\p),0}(\hat \beta_{\bb})}\right) b_0^{-(\p+1),1}(\hat \beta_{\bb}, \hat \alpha)\\
                                                      = &  \frac{1}{4}\int_{-1}^{\z} \Big(\sum_{a\in \cA}\frac{2d_as_ax_a}{1+x_at}\Big)(\z-t)(t+\bb)^{1-\p} p_c(t) dt  \\
                                                         &+ \frac{1}{2}(1+\z)(\bb-1)^{1-\p} p_c(-1) - \frac{1}{4}\int_{-1}^{\z} (\z-t)(t+\bb)^{-(\p+1)}A_2p_c(t) dt \\
                                                     = & \frac{1}{4} G_{x,\bb,\p}(\z) =\frac{1}{4}\Big((\z+\bb)^{1-\p} F_{\Omega,\bb,\p}(\z)\Big).
                                                     \end{split}
                                                     \end{equation*}
The claim follows. \end{proof}
\begin{prop} \label{admissible-stable} {\it Let $M = P(E_0 \oplus E_{\infty})\to S$ be an admissible manifold over a CSCK base, and $L$ an admissible polarization of $M$ which corresponds, up to a scale, to an admissible K\"ahler class $\Omega$. If for some $\bb>1$ $(M, L)$ is $(\hat \beta_{\bb}, \p)$-K-stable, then $\mathfrak{F}_{\Omega, K, \bb, \p}(K)=0$ and $F_{\Omega,\bb,\p}(z)>0$ on $(-1,1)\cap \Q$. }
\end{prop}
\begin{proof} We first prove that if an admissible polarized manifold $(M, L)$ is $(\hat \beta_{\bb}, \p)$-K-stable, then $\mathfrak{F}_{\Omega, K, \bb, \p}(K)=0$, i.e. $A_1=0$, see Proposition~\ref{futaki}. To this end,  we consider the product  test configuration $\mathcal X= M\times \C$ with $\pi: \mathcal X \to \C$ being the projection to the $\C$-factor, and the polarization $\cL= L \otimes \cO_{\C}$. We endow $\mathcal X$ with  the $\C^*$-action given by $\beta$ on $M$ and the standard $\C^*$-action on $\C$, and consider, via $\beta_0$,  the lifted action on $\mathcal{L}$. Thus, the central fibre of this test configuration is $(M,L)$ with induced $\C^*$-action $\alpha=\beta_{0}$.  A computation similar to \eqref{1} shows that \eqref{asymptotics} hold true with
\begin{equation*}
b_0^{-(1+\p), 1}(\hat \beta_{\bb}, \alpha) =\alpha_{1, -(1+\p)}, \qquad  b_1^{1-\p, 1}(\hat \beta_{\bb}, \alpha) = \frac{1}{2} \beta_{1, (1-\p)}.
\end{equation*}
and \eqref{c}. It follows from the definition of ${\rm DF}_{(\hat \beta_{\bb}, \p)}$ and $(\hat \beta_{\bb}, \p)$-K-stability that
\begin{equation*}
\begin{split}
{\rm DF}_{(\hat \beta_{\bb}, \p)} (\mathcal X, \cL) &= \frac{1}{2} \Big(\beta_{1, (1-\p)} - \frac{\alpha_{1, -(1+\p)}}{\alpha_{0, -(1+\p)}} \beta_{0, (1-\p)}\Big) \\
                                                       &= \Big(\frac{\alpha_{0, -(1+\p)}\alpha_{2, -(1+\p)}-\alpha_{1, -(1+\p)}^2}{4\alpha_{0, -(1+\p)}}\Big) A_1\\
                                                       &=0.
                                                       \end{split}
                                                       \end{equation*}
Thus, we have $A_1=0$ or, equivalently, $\mathfrak{F}_{\Omega, K, \bb, \p}(K)=0$, see Proposition~\ref{futaki}.  By Proposition~\ref{futaki-ruled}, the corresponding function  satisfies $F_{\Omega,\bb,\p}(\z)>0$ for $\z\in (-1,1)\cap \Q$.  \end{proof}

\bigskip
By Lemma~\ref{reduced-scalar}, the $(\hat \beta_{\bb}, \p)$-extremal vector field  of an admissible polarized manifold $(M, L)$ (endowed with a maximal torus $\T \subset {\rm Aut}(M,L)$ covering a maximal torus of the isometry group of an admissible K\"ahler metric) is  $A_1K$. As the definition of the relative $(\beta_a, \p)$-Donaldson--Futaki invariant given in Definition~\ref{relative-K-stability} does not change if we replace  $\gamma_{\rm ex}$ by a non-zero multiple, and the inner product \eqref{inner-product}  does not depend on the chosen lift of the action to $L$, we can assume that $\gamma_{\rm ex}$ corresponds the the $\Sph^1$-action $\beta_0$. Calculations similar to the ones in the proofs of Propositions~\ref{futaki-ruled} and \ref{admissible-stable} allow us to  compute the relative $(\hat \beta_{\bb}, \p)$-Donaldson--Futaki invariant of an admissible test configuration
\begin{prop}\label{relative-futaki-ruled} Let $M = P(E_0 \oplus E_{\infty})\to S$ be an admissible manifold over a CSCK base, and $L$ an admissible  polarization of $M$  corresponding, up to a scale, to an admissible K\"ahler class $\Omega$. Then, the relative $(\hat \beta_{\bb}, \p)$-Donaldson--Futaki  invariant with respect to a maximal torus $\T \subset {\rm Aut}(M, L)$ covering a maximal torus of isometries of an admissible K\"ahler metric on $M$ of the test configuration  $(\mathcal X, \cL_{\z})$ with parameter $\z\in \Q \cap (-1, 1)$ corresponding to the degeneration of the normal cone of the infinity section $e_{\infty}=P(0 \oplus E_{\infty})$,  is  a positive multiple of  $F_{\Omega,\bb,\p}(\z)$. In particular, if $(M, L)$ is relative $(\hat \beta_{\bb}, \p)$-K-stable, then $F_{\Omega,\bb,\p}(\z)>0$ on $(-1,1)\cap \Q$.
\end{prop}

\section{Proof of the main results}\label{proofs}
\subsection{Existence results for  admissible $(z+\bb, \p)$-extremal K\"ahler metrics} In this section, we  use Proposition~\ref{emsolution} to  construct admissible $(z+\bb, \p)$-extremal K\"ahler manifolds.
\begin{thm} \label{fextremalexistence}
Suppose that $M=P(E_{0} \oplus E_{\infty}) \rightarrow S$ is an admissible manifold over a CSCK base $S$. Then, for every choice of
$\bb, \p \in {\mathbb R}$ such that $\bb>1$,  $M$ admits an admissible $(z+\bb, \p)$-extremal K\"ahler metric in the admissible K\"ahler class  $\Omega_x$ if the  parameters $x=(x_a, a\in \cA)$ are sufficiently small. If, furthermore, $S$ is a local product of non-negative CSCK metrics, then  any admissible K\"ahler class on $M$ contains an admissible $(z+\bb, \p)$-extremal metric.
\end{thm}
\begin{proof}
It is not hard to check that the linear system in \eqref{alpha-beta-r-q} does not degenerate as $x\rightarrow 0$ (meaning $x_a\rightarrow 0$ for all $a\in \cA$). In particular,
$ \displaystyle\lim_{x\rightarrow 0}A_1$ and $ \displaystyle\lim_{x\rightarrow 0}A_2$ both exist. Further,  the limit for $x\rightarrow 0$ of
the right hand side of \eqref{emNew} is a function of $z$ that has
\begin{itemize}
\item at most one zero in $(-1,1)$ if $d_0=d_\infty=0$.
\item at most two zeroes in $(-1,1)$ if either $d_0$ or $d_\infty$ are non-zero, but not both.
\item at most three zeroes in $(-1,1)$ if both $d_0$ and $d_\infty$ are non-zero.
\end{itemize}
Combined with the end point conditions \eqref{endG} in each of these three cases, $\displaystyle \lim_{x\rightarrow 0} G(z)$ will not have enough inflection points to have any zeroes over $-1<z<1$. Hence we can conclude that for $|x_a|$ sufficiently small for all $a\in \cA$,
$F_{x,\bb, \p}(z)$  satisfies  (i) of \eqref{positivityF} and the claim follows from Proposition~\ref{emsolution}.

\smallskip
Assuming that  that $S$ is a local product of non-negative CSCK metrics,  we shall adapt the root counting argument due to Hwang \cite{Hwang94} and Guan \cite{Guan95} (see also Proposition 11 in \cite{acgt}) to  check that that the function $F_{x, \bb, \p}(z)$  in Proposition~\ref{emsolution}  verifies  (i) of \eqref{positivityF}  for any admissible data $x_a, a\in \cA$, thus defining an admissible  $(z+\bb, \p)$-extremal metric in any admissible K\"ahler class.

The idea is to interpret \eqref{emNew} in a form similar to Equations (4) and (5) of \cite[Prop. 1]{acgt}. Indeed, \eqref{emNew} can be rewritten as follows
\begin{equation}\label{newem1}
(z+\bb)^{\p+1} G_{x,\bb,\p}''(z) = (\prod_{a\in \hat{\cA}} (1+x_a z)^{d_a-1})P(z),
\end{equation}
where $P(z)$ is a polynomial of degree $\leq \# \hat{\cA}+1,$ satisfying for all $a \in  \hat{\cA}$,
\begin{equation}\label{newem2}
P(-1/x_a) =  2 d_a s_a x_a  (-1/x_a + \bb)^2\prod_{b \in \hat{\cA} \setminus \{a\}} \!\!\!\left(1-\frac{x_b}{x_a}\right).
\end{equation}
Now, the  positivity (i) of \eqref{positivityF} is equivalent to positivity of $G_{x,\bb,\p}(z)=\frac{F_{x,\bb,\p}(z)}{(z+\bb)^{\p-1}}$ over $(-1,1)$.
Thus our task is to check that, under the assumptions in the theorem,
$G_{x,\bb,\p}(z)$ is positive for $-1<z<1$. Note that, for $-1<z<1$, the sign of $G_{x,\bb,\p}''(z)$ equals the sign of  \newline
$(z+\bb)^{p+1} G_{x,\bb,\p}''(z)$, and hence the sign of $P(z)$ defined by \eqref{newem1}-\eqref{newem2}. Using the boundary conditions \eqref{endG} for $G_{x,\bb,\p}(z)$, the proof now essentially follows the proof of  \cite[Prop.~11]{acgt} with some minor justifications.
\end{proof}

\subsection{A non-existence result for $(z+\bb, \p)$-CSCK metrics}   In this section we establish the following
\begin{thm}\label{lahdili2} {\it Let  $M$ be an admissible K\"ahler manifold over a CSCK base. Suppose that  $\Omega$ is an admissible K\"ahler class  which is a positive multiple of an element in  $H^2(M, \Z)$, and  the function $F_{\Omega,\bb,\p}(z)$ defined in Proposition~\ref{emsolution} has negative values on $(-1,1)$. Then,  $M$ does not admit a $(z+\bb,\p)$-CSCK metric in $\Omega$.}
\end{thm}
\begin{proof} It was observed in \cite{am,lahdili} that the vanishing of the $(K, \be, \p)$-Futaki invariant $\mathfrak{F}_{(\Omega, K, \be, \p)}$  is a necessary condition for the existence of an $(f, \p)$-CSCK metric in $\Omega$.  In the admissible setting, by using Proposition~\ref{futaki}, this corresponds to the condition that $A_1$ given by \eqref{A1andA2} vanishes.  In the remainder of the argument, we can therefore assume that $\mathfrak{F}_{(\Omega, K, \be, \p)}=0$.   A. Lahdili proved in \cite[Thm.~1]{lahdili2} that if $(M, J, \Omega, K, \T)$ is a compact K\"ahler manifold  as  in Section~\ref{s:mabuchi-futaki},  such that the K\"ahler class  $\frac{1}{2\pi}\Omega \in H^2(M, \Z)$  and  the $(K, \be, \p)$-Futaki invariant $\mathfrak{F}_{(\Omega, K, \be, \p)}$ vanishes on ${\rm Lie}(\T)$,   then  the boundedness from below of the relative $(K, \be, \p)$-Mabuchi functional ${\mathcal M}^{\T}_{(\Omega, K, \be, \p)}$ is a necessary condition for the existence of an $(f, \p)$-CSCK metric in $\Omega$.   As all of the conditions are invariant under a positive scale of $\Omega$, combining Lahdili's result with Proposition~\ref{Lahdili}  concludes the proof. \end{proof}
Notice that in the admissible setting,  the assumption  that a positive multiple of $\Omega_x$ belongs to $ H^2(M, \Z)$ corresponds to admissible data $x=(x_a, a\in \cA)$ such that   $x_a\in \Q, a\in \cA$.

\subsection{A Yau--Tian--Donaldson type correspondence} Because of  Propositions~\ref{futaki}, \ref{admissible-stable} and~\ref{relative-futaki-ruled}, we give the following
\begin{defn}\label{d:admissible-stable} Let $M = P(E_0 \oplus E_{\infty})\to S$ be an admissible manifold over a CSCK base, $L$ an admissible polarization of $M$,  which defines, up to a scale, an admissible K\"ahler class $\Omega$.
\begin{enumerate}
\item[\rm (a)] We say that $(M, L)$ is $(\hat \beta_{\bb}, \p)$-K-semistable/$(\hat \beta_{\bb}, \p)$-K-stable/{\it analytically} $(\hat \beta_{\bb}, \p)$\linebreak -K-stable on {\it admissible test configurations} if $A_1$ given by \eqref{A1andA2} vanishes and, respectively,\linebreak  $F_{\Omega,\bb,\p}(z)\ge 0$ on $(-1,1)$/$F_{\Omega,\bb,\p}(z)>0$ on $(-1,1)\cap \Q$/$F_{\Omega,\bb,\p}(z)>0$ on $(-1,1)$.
\item[\rm (b)]  Similarly,  $(M, L)$ is said to be relative $(\hat \beta_{\bb}, \p)$-K-semistable/relative $(\hat \beta_{\bb}, \p)$-K-stable\linebreak /analytically relative  $(\hat \beta_{\bb}, \p)-$K-stable  on {admissible test configurations} if $F_{\Omega,\bb,\p} \ge 0$ on $(-1,1)$/$F_{\Omega,\bb,\p}(z)>0$ on $(-1,1)\cap \Q$/$F_{\Omega,\bb,\p}>0$ on $(-1,1)$.
\end{enumerate}

Our discussion from the previous sections can be  summarized in the following
\begin{thm}\label{thm:admissible-stable} Let $M = P(E_0 \oplus E_{\infty})\to S$ be an admissible manifold, and $L$ an admissible polarization of $M$ which defines, up to a scale, an admissible K\"ahler class $\Omega$.

$\bullet$   If $(M,L)$ is analytically relative $(\hat \beta_{\bb}, \p)$-K-stable (resp. analytically $(\hat \beta_{\bb}, \p)$-K-stable) with respect to admissible test configurations, then there exists an admissible $((z+ \bb), \p)$-extremal K\"ahler metric in $\Omega$ (resp. an admissible K\"ahler metric of  constant $((z+ \bb), \p)$-scalar curvature).

$\bullet$  If  $\Omega$  admits a K\"ahler metric of constant $((z+ \bb), \p)$-scalar curvature, then $(M,L)$ is $(\hat \beta_{\bb}, \p)$-K-semistable with respect to admissible test configurations.
\end{thm}
\begin{proof} The first statement follows from Proposition~\ref{emsolution}. The second statement follows from \cite[Cor.~2]{am}, Propositions~\ref{futaki} and Theorem~\ref{lahdili2}. \end{proof}
\end{defn}
\begin{remark} In the previous theorem {\em analytic} stability implied the
existence of a distinguished metric.  We do not expect that the same
is implied just by $(\hat \beta_{\bb}, \p)$-K-stability, see \cite{acgt} for an example with $\bb=+\infty$. Below, we prove no such counterexample exists if $p=2m=4$ and $A_1=0$.
\end{remark}
\begin{thm}\label{thm:ruled-classification} Let $(M, J)={P}(\cO \oplus E) \to \Sigma$ be a ruled complex surface over a compact complex curve, where  $E$ is a line bundle of positive degree over $\Sigma$, $L$ a  polarization of $(M, J)$,  which, up to a positive scale, corresponds to an admissible K\"ahler class $\Omega_x$ with $x\in(0,1)$, and $\bb >1$ a real number.  Then the following conditions are equivalent
\begin{enumerate}
\item[\rm (i)] $\Omega_x$ admits an admissible K\"ahler metric  which is conformally Einstein--Maxwell with conformal factor $(z+\bb)^{-2}$;
\item[\rm (ii)]  $\Omega_x$ admits a K\"ahler metric which is conformally Einstein--Maxwell with conformal factor $(z+\bb)^{-2}$;
\item[\rm (iii)] $(M, L)$ is $(\hat \beta_{\bb}, 4)$-K-stable on admissible test configurations;
\item[\rm (iv)] $(M, L)$ is analytically $(\hat \beta_{\bb}, 4)$-K-stable on admissible test configurations.
\end{enumerate}
\end{thm}
\begin{proof}We first notice that in this case, $F_{\Omega_x,\bb,\p}(z)=F_{x,\bb, \p}(z)$ is a polynomial of degree $\le 4$. (This follows from Proposition~\ref{emsolution} and \eqref{G}, see also \cite{KoTo16}.) As shown in \cite[Cor.~2]{lahdili2}, the construction of  \cite{KoTo16} combined with Proposition~\ref{Lahdili} above  and the stability under deformation result in \cite{lahdili},  yield the equivalences
$${\rm (i)} \Longleftrightarrow {\rm (ii)} \Longleftrightarrow {\rm (iv)}.$$
Noting that clearly ${\rm (iv)} \Longrightarrow {\rm (iii)}$, we thus need to establish the implication ${\rm (iii)} \Longrightarrow {\rm (iv)}$, i.e. that
$$  \ F_{x,\bb, 4}(z)>0 \ {\rm on} \ (-1,1) \cap \Q  \Longrightarrow F_{x, \bb, 4}(z)>0 \ {\rm on} \ (-1,1)$$
under the assumption $\mathfrak{F}_{\Omega_x, \bb, 4}(K)=0=A_1$ (with $A_1$ computed by \eqref{A1andA2}).

Suppose for contradiction that $F_{x, \bb, 4}(z) \ge 0$ on $(-1,1)$ and has double irrational root. By Theorem \ref{fextremalexistence} this implies that the genus of $\Sigma$ is at least two. By the results in \cite{KoTo16}, the condition $A_1=0$ implies that $\bb=\bb(x)$ is a root of
\begin{equation}\label{a-x}
x=2\bb/(1+\bb^2)
\end{equation}
which is determined uniquely by the requirement $|\bb|>1$.
On the other hand, the condition that  $F_{x, \bb, 4}(z)$ has a double root on $(-1,1)$, together with \eqref{positivityF}(ii),  implies the vanishing of the discriminant of the second order polynomial $F_{x, \bb, 4}(z)/(1-z^2)$.  It then follows  (see \cite[Sect. 3.1]{KoTo16})  that  $x$ must coincide with the a root in $(0,1)$ of
\begin{equation}\label{x0}
D_s(x) = 12 + 12 s x - 19 x^2 - 12 s x^3 + (7+s^2) x^4 + 6(2+2 s x - 2 x^2 - s x^3)\sqrt{1-x^2}=0,
\end{equation}
with $s = 2(1-{\bf g})/d$  (here ${\bf g}$ stands for the genus of $\Sigma$ and $d$ for the degree of $L$).  We denote the above two roots by $(x_0, \bb_0)$.

Now, we can take $x=x_0$ to be rational, as the polarization assumption implies that $\Omega_{x}$ has rational coefficients. Our first goal is to show $\bb_0$ must also be rational. A careful  look at  \eqref{x0} reveals that  either $\sqrt{1-x_0^2}$  is rational, and then so must $\bb_0=(1+\sqrt{1-x_0^2})/x_0$ be,  or else $f_1(x_0) = f_2(x_0) =0$, where $f_1=12 + 12 s x - 19 x^2 - 12 s x^3 + (7+s^2) x^4$ and $f_2 = 6(2+2 s x - 2 x^2 - s x^3)$. However, the latter cannot hold since $f_2=0$ implies $s=\frac{2 (1-x^2)}{x (-2+x^2)}$ and substituting this into $f_1- f_2=0$ gives
$\frac{(1-x^2) x^4 (-12+7 x^2)}{(-2+x^2)^2}=0$ which cannot  be zero for  $0<x<1$. Thus $\bb_0$
is rational.

As both $\bb_0$ and $x_0$ are rational, it follows from the explicit computations in \cite{KoTo16} that $F_{x_0, \bb_0, 4}(z)$ has rational coefficients.  As  $\pm 1$ are roots of $F_{x_0, \bb_0, 4}(z)$ by \eqref{positivityF}(ii)-(iii), any double root of $F_{x_0, \bb_0, 4}(z)$ must then be rational too, a contradiction.\end{proof}

\section{Conformally K\"ahler, Einstein--Maxwell metrics}\label{s:EM}

While Proposition \ref{emsolution} gives an existence result for the boundary value problem consisting of \eqref{em1} together with (ii) and (iii) of \eqref{positivityF}, if we are aiming for  $(|z+\bb|,\p)$-CSCK solutions, we need to set $A_1=0$, which  in turn yields an equation for $\bb$ in terms of the admissible data that may or may not have a solution with $|\bb|>1$. Of course, if we are successful in finding such a solution, we still need to ensure that  (i) of \eqref{positivityF} holds. We can often be more specific about the form of the solution when  $A_1=0$, if we consider particular values of $\p$.

Indeed, supposing $\p\neq 0,1,...,m,m+1$ then it is easy to see that any solution to \eqref{em1} is of the form
\begin{equation}\label{em2}
F(z) = c_\p(z+\bb)^\p + c_{\p-1}(z+\bb)^{\p-1} + \sum_{k=0}^{m}c_k(z+\bb)^k,
\end{equation}
where $c_0,...,c_m$ depend on $A_1$, $A_2$ and the admissible data. Thus,  $F_{x,\bb,\p}(z)$  must be given by \eqref{em2} for a unique choice of  the coefficients $A_1$, $A_2$, $c_\p$, and $c_{\p-1}$ (whose existence is guaranteed by Proposition~\ref{emsolution} so that $F(z)$ satisfies (ii) and (iii) of \eqref{positivityF}).
In particular, if $\p>m+1$ is an integer, then $F(z)=F_{x,\bb,\p}(z)$ is a polynomial.

In this section,  we will explore further the case  $\p=2m$  ($m>1$). Then  \eqref{em1} becomes

\begin{multline}\label{p=2m}
-(z+\bb)^2 F''(z) + 2(2m-1)(z+\bb) F'(z) - 2m(2m-1) F(z)\\
 =
 (A_1 z+A_2)\Mpc(z) - \Mpc(z)  (z+\bb)^2 \sum_{a \in \hat{\cA}}
\frac{2 d_a s_a x_a }{1+x_az}.
\end{multline}

Imposing $A_1=0$ and (i) of \eqref{positivityF} in \eqref{p=2m}  will thus  produce an Einstein--Maxwell  metric $h= \frac{1}{(z+\bb)^2} g$ where $g$ is given by \eqref{gg} with $\Theta(z)=F(z)/p_c(z)$, see \cite{am}.

\subsection{Conformally K\"ahler, Einstein--Maxwell metrics over the product of two Riemann surfaces}

Let $\Sigma_a$ $(a=1,2)$ be compact Riemann surfaces with CSCK metrics $(\pm
g_a,\pm \omega_a)$ and let $M$ be $P(\cO\oplus E)\to \Sigma_1\times\Sigma_2$
where $E=E_1\otimes E_2$ for $E_a$ being pullbacks of line bundles on
$\Sigma_a$ with $c_1(E_a)=[\omega_a/2\pi]$. Let $\pm 2s_a$ be the scalar
curvature of $\pm g_a$. Note that if $\deg E_a = n_a$, then $s_a = 2(1-{\bf g}_a)/{n_a}$, where
${\bf g}_a$ denotes the genus of $\Sigma_a$. In this case, with slight abuse of notation, we will also write $M = P(\cO\oplus \cO(n_1,n_2))\to \Sigma_1\times\Sigma_2$.

Equation \eqref{p=2m} now takes the form
\begin{multline}\label{em3}
-(z+\bb)^2 F''(z) + 10(z+\bb) F'(z) - 30 F(z)\\
 =
(A_1 z+A_2)(1+x_1 z)(1+x_2 z) +  (z+\bb)^2\left( 2s_1 x_1(1+x_2z) +2s_2 x_2(1+x_1z)\right),
\end{multline}

and we know that its solution $F_{x,\bb, 6}(z)$ is of the form
$$F_{x,\bb, 6}(z) = c_{6} (z+\bb)^{6} + c_{5} (z+\bb)^{5} + c_3(z+\bb)^3 + c_2(z+\bb)^2 + c_1 (z+\bb) + c_0.$$
Plugging $F_{x,\bb, 6}(z)$ into \eqref{em3} tells us that
$$
\begin{array}{ccl}
c_0 &=& \frac{\bb(\bb x_1-1)(\bb x_2-1)A_1 - (\bb x_1-1)(\bb x_2-1) A_2}{30}\\
\\
c_1 & = & \frac{(-1 + 2\bb  x_1 + 2\bb x_2-3\bb^2x_1x_2)A_1 + (-x_1-x_2+2\bb x_1x_2)A_2}{20}\\
\\
c_2 & = & \frac{(-x_1-x_2+3\bb x_1x_2)A_1 - x_1x_2 A_2 +2(s_1 x_1(1-\bb x_2) +  s_2x_2 (1-\bb x_1))}{12}\\
\\
c_3 & = &\frac{-x_1x_2 A_1 + 2x_1x_2(s_1+s_2)}{6}
\end{array}
$$
On the other hand, (ii) of \eqref{positivityF} is equivalent to
$$
\begin{array}{ccl}
c_5 &=& \frac{-(2\bb (3+\bb ^2)(1+3\bb^2)c_0 +(\bb^2-1)(1+10\bb ^2+5\bb ^4)c_1 + 4\bb (\bb^2-1)^2(\bb^2+1) c_2 + (\bb^2-1)^3(1+3\bb^2)c_3)}{(1-\bb^2)^5} \\
\\
c_6 & = &  \frac{(1+10\bb^2+5\bb^4)c_0 +4\bb(\bb^4-1)c_1 + (\bb^2-1)^2(1+3\bb^2)c_2+ 2\bb(\bb^2-1)^3 c_3}{(1-\bb^2)^5}, \\
\\
\end{array}
$$
so $c_5$ and $c_6$ are determined by $c_0,c_1,c_2,c_3$ (and $\bb$). With this established,  (iii) of \eqref{positivityF}  is equivalent to the following two equations
\begin{equation}\label{(iii)eqn}
\begin{array}{cl}
&(3+12 \bb+30 \bb^2+20 \bb^3+15 \bb^4)c_0  +2 (\bb^2-1) (1 + 5 \bb + 5 \bb^2 + 5 \bb^3) c_1   \\
\\
+&2 (\bb^2-1)^2  (1 + 2 \bb + 3 \bb^2)c_2+(\bb^2-1)^3 (1 + 3 \bb)c_3\\
\\
=& (\bb-1)(\bb+1)^5 ( x_1-1) (x_2-1)\\
\\
\\
&(3 - 12 \bb + 30 \bb^2 - 20 \bb^3 + 15 \bb^4)c_0+  2 (\bb^2-1) (-1 + 5 \bb - 5 \bb^2 + 5 \bb^3)c_1\\
\\
+& 2  (\bb^2-1)^2 (1 - 2 \bb + 3 \bb^2) c_2 +  (\bb^2-1)^3 (-1 + 3 \bb)c_3\\
\\
= &(\bb-1)^5(\bb+1) (1 + x_1) (1 + x_2).
\end{array}
\end{equation} Using the formulas for $c_0,...,c_3$ above this yields a linear system of two equations with the two unknowns $A_1$ and $A_2$.
The linear system has coefficients that depend on the admissible data ($s_1,s_2,x_1,x_2$) as well as $\bb$ in a rather unwieldy way.

One may check, aided by Mathematica, that if $x_1,x_2>0$ then there exists $\bb>1$ such that $A_1=0$.
In general, it appears very  to be quite difficult to write this $\bb>1$ solution explicitly,  so it is  non-trivial to test the final condition (i) of \eqref{positivityF}.
But we were able to describe an explicit example below.

\begin{ex}\label{negative-scal}
Let $x_1=1/2$, $x_2=1/3$,  $\bb=5$.
This gives
$$
\begin{array}{ccl}
A_1&=& \frac{20(9840-4502s_1+1203s_2)}{24073}\\
\\
A_2&=& \frac{20(7836+4442s_1+2883 s_2)}{3439}\\
\\
c_0 &=& -\frac{12 (314+2978 s_1+787 s_2)}{24073}\\
\\
c_1 & = & \frac{2 (-135+1294 s_1+279 s_2)}{1267}\\
\\
c_2 & = & \frac{7640-14198 s_1-2697 s_2}{15204}\\
\\
c_3 & = &\frac{-98400+69093 s_1+12043 s_2}{433314}\\
\\
c_5&=& \frac{415-83 s_1-13 s_2}{30408}\\
\\
c_6&=& \frac{-21912+2862 s_1+433 s_2}{13866048}.
\end{array}
$$
When we solve for $A_1=0$ we get $s_2=\frac{2(2251s_1-4920)}{1203}$
and then compute
$Scal(h)=\frac{240}{401}(148 s_1-153)$ and {\small
$$F_{x,\bb, 6}(z) = \frac{(1-z^2)}{86616}\left( 3 (26078+22965 z+7553 z^2+1095 z^3+53 z^4)+s_1(1-z^2)(1181 + 465 z + 28 z^2)\right).$$}
Notice that $F_{x, \bb, 6}(z)$ satisfies (i) of \eqref{positivityF} when $s_1>0$, so we get a family of conformally K\"ahler,  Einstein--Maxwell metrics on
$$M= P(\cO\oplus \cO(n_1,n_2))\to {\mathbb C}{\mathbb P}^1 \times\Sigma_2,$$ where
$n_1 \in {\mathbb Z}^+$ is arbitrary and
${\bf g}_2, n_2\in {\mathbb Z}^+$ are such that
$$\frac{(4502-4920n_1)}{1203n_1} = \frac{(1-{\bf g}_2)}{n_2}.$$
We can take for instance  $n_2= 1203n_1$ and ${\bf g}_2 = 4920n_1-4501$ in order to satisfy the above relation.
With this choice,  $Scal(h)>0$ for $n_1=1$, but for $n_1>1$ $Scal(h)<0$.

Since $F_{x,\bb, 6}(z)$ satisfies (i) of \eqref{positivityF} when $s_1=0,$ we also have conformally K\"ahler,  Einstein--Maxwell metrics on
$M= P(\cO\oplus \cO(n_1,n_2))\to T^2  \times\Sigma_2,$ where
$n_1 \in {\mathbb Z}^+$ is arbitrary and
${\bf g}_2, n_2\in {\mathbb Z}^+$   are such that
$$\frac{(-1640)}{401} = \frac{(1-{\bf g}_2)}{n_2}.$$
A specific solution is $n_2= 401, {\bf g}_2 = 1641$.

Finally,  if $s_1<0$ is sufficiently close to zero, (i) of \eqref{positivityF} still holds and we will have some conformally K\"ahler,  Einstein--Maxwell metrics on
$M= P(\cO\oplus \cO(n_1,n_2))\to \Sigma_1  \times\Sigma_2,$
where the genus of $\Sigma_1$ and $\Sigma_2$ are both at least two. (On the other hand,  \eqref{positivityF}-(i) fails as $s_1 \rightarrow -\infty$.)

By \cite[Thm. 8]{acgt},  none of the manifolds above admits a  CSCK metric.
\end{ex}

The next example is inspired by the  construction of K\"ahler--Einstein admissible metrics  by Koiso and Sakane  \cite{koi-sak1,sakane}.

\begin{ex}\label{e:koiso-sakane}
We consider $M= P(\cO\oplus \cO(1,-1))\to {\mathbb C}{\mathbb P}^1 \times{\mathbb C}{\mathbb P}^1$. Thus,
we assume that $s_1=2$ and $s_2=-2$, $0<x_1<1$ and $-1<x_2<0$. The pair $(x_1,x_2) \in (0,1)\times(-1,0)$ determines the admissible K\"ahler class and, up to rescaling, this exhausts the entire K\"ahler cone.
Notice that when $x_1=1/2=-x_2$, the corresponding K\"ahler class admits a K\"ahler--Einstein admissible metric which was first discovered by
Koiso--Sakane. Moreover, for $x_2=-x_1$ or $x_2=-1+x_1$, the corresponding K\"ahler class admits a CSCK admissible metric (see e.g.  \cite[Thm. 9]{acgt}).

As in Section \ref{hodge4}, this is a case that is included by Theorem \ref{fextremalexistence}. For a given pair $(x_1,x_2) \in (0,1)\times(-1,0)$, by Proposition \ref{emsolution}, we have an admissible metric associated to a conformally K\"ahler Einstein-Maxwell metric whenever there is a solution $|\bb|>1$ of $A_1 =0$.
Here we calculate that $A_1 =0$ if and only if
$$
\begin{array}{ccl}
q(x_1,x_2,\bb)& : = & -3 (x_1+x_2) (x_1-x_2+x_1 x_2)\\
\\
& & +3 \left(2 x_1+3 x_1^2-2 x_2+8 x_1 x_2+2 x_1^2 x_2+3 x_2^2-2 x_1 x_2^2+2 x_1^2 x_2^2\right)\bb \\
\\
& & - 3 (x_1+x_2) (15-2 x_1+2 x_2+17 x_1 x_2)\bb^2\\
\\
&& + (60-10 x_1+45 x_1^2+10 x_2+240 x_1 x_2-18 x_1^2 x_2+45 x_2^2+18 x_1 x_2^2+90 x_1^2 x_2^2)\bb^3\\
\\
& & - 5 (x_1+x_2) (33+4 x_1-4 x_2+45 x_1 x_2)\bb^4\\
\\
& & + (72+34 x_1+123 x_1^2-34 x_2+408 x_1 x_2+50 x_1^2 x_2+123 x_2^2-50 x_1 x_2^2+90 x_1^2 x_2^2)\bb^5 \\
\\
& & -(x_1+x_2) (159-2 x_1+2 x_2+105 x_1 x_2)\bb^6\\
\\
& & +(60-30 x_1+15 x_1^2+30 x_2+96 x_1 x_2-38 x_1^2 x_2+15 x_2^2+38 x_1 x_2^2+6 x_1^2 x_2^2)\bb^7\\
\\
& & + 15 (-1+x_1-x_2) (x_1+x_2)\bb^8  =0.
\end{array}$$

Now, $q(x_1,x_2,1) = 192(-1+x_1)^2(-1+x_2)^2>0$ and $q(x_1,x_2,-1)=-192(1+x_1)^2(1+x_2)^2<0$.
If the leading coefficient of $q(x_1,x_2,\bb)$, i.e. $15(x_1+x_2)(-1+x_1-x_2)$, is nonzero, it
follows that there is  an $|\bb|>1$ such that $q(x_1,x_2,\bb)=0$.
On the other hand, when $x_2=-x_1$, we have
$$
\begin{array}{ccl}
q(x_1,x_2,\bb)/\bb&  = & 96(1-x_1^2)^2\\
\\
&+& 16 (1-x_1^2)(12-2x_1+9x_1^2) (\bb^2-1)\\
\\
&+& 2 \left((1-x_1)(59+31x_1-59 x_1^2-27x_1^3)+4\right)(\bb^2-1)^2\\
\\
&+& \left((1-x_1)^2(28+26 x_1 - 9x_1^2-6x_1^3)+2+6x_1^5\right)(\bb^2-1)^3
\end{array},$$
and so in this case $q(x_1,x_2,\bb)=0$ has no solutions for $|\bb|>1$.
In a similar, but slightly more tedious fashion one can verify that when $x_2=-1+x_1$, there are no solutions $|\bb|>1$ to $q(x_1,x_2,\bb)=0$. However, in those cases there are  admissible CSCK metrics (see \cite{acgt}),  which are of course conformally K\"ahler, Einstein--Maxwell metrics with constant conformal factor (in the setting of this paper, they can be thought as  $((z+\bb), 2m)$-CSCK  metrics with $\bb = \infty$),  so we can  formulate the following general existence result.

\begin{prop}\label{O(-1,1)}
Every K\"ahler class on $M= P(\cO\oplus \cO(1,-1))\to {\mathbb C}{\mathbb P}^1 \times{\mathbb C}{\mathbb P}^1$ has an admissible K\"ahler metric, conformal to   an Einstein--Maxwell metric.
\end{prop}

\end{ex}

\subsection{Conformally K\"ahler, Einstein--Maxwell metrics on admissible projective bundles over a CSCK $4$-manifold}\label{hodge4}
We now assume that the base $S$  is a Hodge  K\"ahler manifold of real dimension $4$ with  non-negative constant scalar curvature $4s$ (we drop the index $a$ when $|{\mathcal A}|=1$).
Note that $0\leq s \leq 3$ (the latter inequality follows from the Fujita inequality for Hodge manifolds, see \cite{fujita}). By \cite[Thm~7]{acgt},  there are no admissible CSCK metrics in this case.

Given $0<x<1$ (again, without loss of generality we may assume $x$ is positive as long as we also do not assume upfront that $\bb>1$), Theorem~\ref{fextremalexistence} and Proposition~\ref{emsolution} tells us that an admissible metric associated to
a $(|z+\bb|,\p)$-CSCK solution is equivalent to a solution $|\bb|>1$ of $A_1=0$.

From \eqref{p=2m} and \eqref{positivityF} it is straightforward to calculate that
$$A_1= \frac{8 \left(-x\bb^2 +2 \bb-x\right)}{45 (\bb-1)^{10} (1+\bb)^{10}}q(\bb,x), $$
where
$$
\begin{array}{ccl}
q(\bb, x) & =&  96 (1 - x)^3\\
\\
&+& 32 (1 - x)^2 (12 - 9 x - s x)(\bb-1)\\
\\
&+&  8 (1 - x) (87 - 120 x - 16 s x + 39 x^2 + 10 s x^2)(\bb-1)^2\\
\\
&+&  8 (1 - x) (93 - 81 x - 29 s x + 18 x^2 + 3 s x^2) (\bb-1)^3\\
\\
&+& 2 (243 - 315 x - 104 s x + 99 x^2 + 70 s x^2 - 15 x^3 + 22 s x^3)(\bb-1)^4\\
\\
&+&  2 (90 - 63 x - 45 s x + 14 s x^2 - 3 x^3 + 15 s x^3)(\bb-1)^5\\
\\
&+&5 (6 - 2 s + s(2+x)(1-x)^2)(\bb-1)^6
\end{array}
  $$
We always get at least one solution of $A_1=0$  with $|\bb|>1$ from the factor $$-x\bb^2 +2 \bb-x=0,$$
namely
\begin{equation}\label{a1-x}
 \bb_0(x) : = \frac{1+\sqrt{1-x^2}}{x}.
 \end{equation}
Notice that  $\bb_0(x)>1$ and  $\bb_0(x)$ is a decreasing function of $0<x<1$ with $\displaystyle \lim_{x\rightarrow 0}\bb_0(x)=+\infty$ and
$\displaystyle \lim_{x\rightarrow 1}\bb_0(x)=1$. Any additional solutions of $A_1 =0$
would come from solutions of the equation $q(\bb,x)=0$ satisfying $|\bb|>1, x\in(0,1)$.

\begin{prop}\label{hodge}
Let $(S, g_S, \omega_S)$ be a compact CSCK Hodge $4$-manifold with non-negative scalar curvature,  and $E$ be a holomorphic line bundle such that
$c_{1}(E)= [\omega_{S}/2\pi]$. Then,  in each admissible K\"ahler class on $M=P(\cO \oplus E) \rightarrow S$ there exists at least one
admissible K\"ahler metric  conformal to an Einstein--Maxwell metric.
\end{prop}

\begin{remark}\label{multiplesolnoverHodge4mnf}
Over the interval  $1<s\leq 3$, the expression $2 s/(1 + s^2)$ is a decreasing function surjecting onto   $[3/5,1)$.
If $x= 2 s/(1 + s^2)$, then we observe that
\begin{multline*}
q(\bb, x)
= \frac{2 (\bb-s)^2}{\left(1+s^2\right)^3}
\left(15-2 s^2+3 s^4 -32 s \left(3+s^2\right)\bb+2 \left(9+130 s^2+5 s^4\right)\bb^2\right.\\
\left.-32 s \left(3+5 s^2\right)\bb^3+5 \left(3+6 s^2+7 s^4\right)\bb^4\right)
=0
\end{multline*}
has a double root at $\bb=s$ and, moreover,  $\bb_0(x)=s$. Thus $A_1=0$ has a triple root at $\bb_0=s$.

If, on the other hand, $2 s/(1 + s^2)<x<1$, then $\bb_0(x) <s$ so that
$$q(\bb_0(x), x) =\frac{6\bb_0 (\bb_0^2-1)^4   \left(5 \bb_0^2-1\right) (\bb_0-s)}{\left(\bb_0^2+1\right)^3}$$
(with $\bb_0 = \frac{1+\sqrt{1-x^2}}{x}$)
is negative
whereas $q(1,x)$ and $\displaystyle\lim_{\bb \rightarrow +\infty}q(\bb,x)$ are positive. We conclude that in this case, for each $x\in (0,1)$
$q(\bb, x)=0$ has (at least) two additional solutions, $\bb_{\pm}(x)$ with  $1<\bb_{-}(x)<\bb_0(x)<\bb_{+}(x)$.

\end{remark}

\begin{remark}
Assuming more generally that $(S,\pm g_S,\pm \omega_S)$ is a Hodge K\"ahler manifold of complex dimension $d$ with non-negative constant scalar curvature,  it seems from experimental data  (letting $d$ take various values $\ge 3$) that we always have the solution $\bb_0(x)$ of $A_1=0$ defined  in \eqref{a1-x} but  a direct  proof of this seems out of reach at the moment.
 \end{remark}

\begin{conjecture}\label{hodgebaseconjecture}  Let $(S, g_S, \omega_S)$ be a compact CSCK Hodge $2(m-1)$-manifold with non-negative scalar curvature and $E$ be a holomorphic line bundle such that $c_{1}(E) = [\omega_{S}/2\pi]$. Then,  in each admissible K\"ahler class on $M=P(\cO \oplus E) \rightarrow S$ there exists at least one admissible K\"ahler metric  which is conformal to an Einstein--Maxwell metric.
\end{conjecture}

\subsection{Conformally K\"ahler, Einstein metrics}\label{conf-Ein}  We recall here the constructions going back to  Page~\cite{page} and B\'erard-Bergery~\cite{lbb} of admissible K\"ahler manifolds which are conformally Einstein. These of course are special cases of the conformally K\"ahler, Einstein--Maxwell metrics discussed in this paper. By the results of \cite{dm,dm1}, any compact K\"ahler manifold  $(M, J, g, \omega)$ of real dimension $2m\ge 6$, which is conformally Einstein  is isometric to one of these examples. We use the computations of \cite[Sect. 5.6]{haml} in order to recast the construction of \cite{lbb,dm} in the admissible setting of this paper.  Indeed, according to \cite[Sect. 1.4 \& 5.6]{haml}, for an admissible  K\"ahler metric  $g$ of the form \eqref{gg}  to be conformally Einstein with a conformal factor $(z+\bb)$, we must have that \begin{enumerate}
\item[$\bullet$] $|\cA|=1$;
\item[$\bullet$] $M=P(\cO \oplus  E)\to S,$  where $(S, g_S, \omega_S)$ is a compact K\"ahler--Einstein manifold of positive scalar curvature $Scal(g_S)= 2(m-1)s$, and $E$ is a  holomorphic line bundle over $S$ with $c_1(E) = [\omega_S/2\pi]$;  here $s= \frac{c_1(S)\cdot c_1(E)^{m-2}}{c_1(E)^{m-1}}$ is the normalized scalar curvature of $g_S$;
\item[$\bullet$]  there exists an admissible parameter $x=x_e\in (0,1)$  and a real constant  $\bb_e>1$,  such that $F(z)= F_{x_e, \bb_e, 2m}(z)$  is given by
\begin{equation}\label{einstein}
\begin{split}
\frac{F_{x_e, \bb_e, 2m}(z)}{x_e^{m-1}} =   & \sum_{j=1}^m \frac{j}{m}  {2m \choose m+j}\Big[\lambda_+ \Big(\bb_e - \frac{1}{x_{e}}\Big)^{m-j} \Big(z+ \frac{1}{x_{e}}\Big)^{m+j}  \\
                                                              & - \lambda_- \Big(\bb_e - \frac{1}{x_{e}}\Big)^{j-1} \Big(z + \frac{1}{x_e}\Big)^{m-j} + \frac{s}{m} \Big(z+ \frac{1}{x_{e}}\Big)^m \Big],
\end{split}
\end{equation}
where $\lambda_+, \lambda_-$ are real constants.
\end{enumerate}
The point is that $F(z)= F_{x_e, \bb_e, 2m}(z)$ automatically verifies \eqref{em1} and $A_1=0$  (because the metric $(z+\bb_e)^{-2}g$ is Einstein and  therefore  $g$ has constant $(z+\bb_e, 2m)$-scalar curvature), so we are left with the 4 boundary conditions \eqref{positivityF}-(ii) \&(iii). These in turn place 4 algebraic relations for the real constants $(\lambda_+, \lambda_-, x_{e}, \bb_e)$.  The upshot of the constructions in \cite{lbb,dm1} is that if $s>1$, then these relations determine the $4$ constants,  up to a two-fold ambiguity, i.e. there exists a  unique  $x_e\in (0,1)$ and a pair  $\bb_e=\bb_{\pm}>1$ for which  $F_{x_e, \bb_{+},2m}(z)= F_{x_e, \bb_-, 2m}(z)$ satisfies \eqref{einstein} (see \cite{dm1,haml} for the geometric meaning of this). Notice that the positivity condition \eqref{positivityF}(i) is  then automatically satisfied by Theorem~\ref{fextremalexistence}.

It is clear from the setting above that we can weaken the K\"ahler--Einstein assumption for $(S,g_S, \omega_S)$ and assume instead that $(S, g_S, \omega_S)$  is a CSCK Hodge manifold with normalized scalar  curvature  $s >1$.  Then the solutions $(x_e, \bb_e= \bb_{\pm})$   will correspond to  two Einstein--Maxwell  metrics $h_{\pm}= (z+ \bb_{\pm})^{-2}g$ in the conformal class of the admissible metric corresponding to $F(z)=F_{x_e, \bb_{\pm}, 2m}(z)$. We thus have the following existence result (related to Conjecture~\ref{hodgebaseconjecture}).
\begin{prop}\label{E-extended} Let $(S, g_S, \omega_S)$ be a compact Hodge K\"ahler $2(m-1)$-manifold  of constant  scalar curvature $Scal(g_S)>2(m-1)$,
and $E$ a holomorphic line bundle such that
$c_{1}(E)= [\omega_{S}/2\pi]$. Then, $M=P(\cO \oplus E) \rightarrow S$  admits an admissible K\"ahler metric  conformal to an Einstein--Maxwell metric.
\end{prop}
Notice that the constraint $Scal(g_S)>2(m-1)$ in the above proposition is equivalent to
$$c_1(S) \cdot c_1(E)^{m-2} > c_1(E)^{m-1},$$
 which in turn  limits the choice for the line bundles $E$ on a given $S$. By the Fujita inequality~\cite{fujita}, the number of  such line bundles cannot exceed $(m-1)$.

\subsection{Conformally K\"ahler,  Einstein--Maxwell metrics and the Yamabe functional}\label{futakifun}
On a compact manifold $M$ of real dimension $2m$,  the  {\em normalized Einstein--Hilbert functional}
on the set of Riemannian metrics  is defined by
$${\mathfrak S}(g):= \frac{\int_M Scal(g)\, dv_g}{(\int_M dv_g)^{\frac{m-1}{m}}},$$
where $dv_g$ denotes the volume form of $g$.
The restriction of
${\mathfrak S}$ to a conformal class $[g]$ of Riemannian metrics on $M$ is known as the {\em Yamabe functional}. It is  a deep result that the Yamabe functional  attains a minimum $Y_{[g]}$ on  $[g]$ (see e.g.~\cite{A98, LePa87, S84}).   Any metric $h \in [g]$ for which
${\mathfrak S}(h)=Y_{[g]}$ is called a {\em Yamabe minimizer} of $[g]$. It is well-known that any Yamabe minimizer $h$ has constant scalar curvature and,  if $Y_{[g]} \leq 0$,  any metric in $[g]$ which has constant scalar curvature must be homothetic to the (unique up scaling) Yamabe minimizer in $[g]$. For $Y_{[g]} >0$, the Yamabe minimizers are not necessarily homothetic and,  furthermore,  a constant scalar curvature metric in $[g]$ is not necessarily a Yamabe minimizer. Thus,   one can ask
\begin{Question}\label{yamabe} Given a constant  scalar curvature metric $h\in [g]$ with ${\mathfrak S}(h) >0$,  is $h$  a Yamabe minimizer?
\end{Question}

It is known (see e.g. \cite{A98}) that if a constant scalar curvature representative $h$ of $[g]$ is a Yamabe minimizer, it must satisfy the inequality \begin{equation}\label{aubin} {\mathfrak S}(h) \leq 2m(2m-1)Vol(\Sph^{2m})^{1/m}, \end{equation} where $\Sph^{2m}$ denote the unit sphere in $\R^{2m+1}$. Furthermore, by~\cite{S84}, the inequality  \eqref{aubin} is  strict if  $(M,[g])$ is not conformal to $\Sph^{2m}$.
Notice that for e.g. $m=2$, the right hand side of this inequality is equal to $8\sqrt{6}\pi$.

\smallskip
In what follows, we shall investigate Question~\ref{yamabe}  for some of the Einstein--Maxwell metrics (which are of constant scalar curvature by definition, see \cite{am}) that we found in the conformal classes of admissible K\"ahler metrics.

As we have already mentioned, some of the Einstein--Maxwell metrics we found have negative scalar curvature (see Example~\ref{negative-scal}), so they are Yamabe minimizers by the above general remarks. Another  such examples  are  the  Einstein metrics discussed in Section~\ref{conf-Ein}, which have positive constant scalar curvature and are Yamabe minimizers by virtue of the Obata theorem~\cite{obata}.

We shall now give examples for which the Einstein--Maxwell metrics are not Yamabe minimizers. To this end, we use the following general remarks. Suppose we have an admissible metric $g$ as defined in \eqref{gg}. Then for any $t>1$ we  consider the conformal metric $h_t := (z+t)^{-2} g$ with scalar curvature equal to
\begin{equation}\label{scalht}
\begin{array}{ccl}
Scal(h_t) & = & \frac{-(z+t)^2 F''(z) + 2(2m-1)(z+t) F'(z) - 2m(2m-1) F(z)}{\Mpc(z)} \\
\\
&+ &(z+t)^2 \sum_{a \in \hat{\cA}}
\frac{2 d_a s_a x_a }{1+x_az},
\end{array}
\end{equation}
and volume form $(z+t)^{-2m}\omega^m/m!$.
It follows  from \cite{futakiono17a} (this can be checked directly in the admissible setting) that the function $f(t)={\mathfrak S}(h_t)$
does not depend on the choice of  $F(z)$, i.e. on the particular choice of admissible representative in the given admissible K\"ahler class.
Moreover, as follows from Theorem 2.3 (b) in \cite{futakiono17a}, (or can be checked directly in the admissible setting) the critical values of $f(t)$ correspond exactly to the values $t=\bb$ where $A_1$ from \eqref{em1} (or equivalently the Futaki invariant $\mathfrak{F}_{([\omega], K, \bb, 2m)}$, see Proposition~\ref{futaki}) vanishes. Thus,  from Proposition \ref{emsolution} it follows that any critical value $t=\bb$ of $f(t)$ corresponds to a conformally K\"ahler,  Einstein--Maxwell metric $h_{\bb} = (z+\bb)^{-2} g$ provided that  $F_{[\omega],\bb,2m}(z)$ satisfies (i) of \eqref{positivityF}. Notice that the latter condition is automatic on the manifolds described in Theorem~\ref{fextremalexistence}. Further, for such an  $h_{\bb}$ to be a Yamabe minimizer, it is necessary (albeit not sufficient) that $f(t)$ has a minimum at $t=\bb$.

\subsection{Einstein--Maxwell metrics  on the first Hirzebruch surface which are not Yamabe minimizers}\label{YamabeFun}
We now restrict ourselves to the case $M=P(\cO \oplus \cO(1)) \to \C P^1$ which has been  studied  in \cite{Le15,KoTo16}. We let $g$ be an admissible
K\"ahler metric given by \eqref{gg}. In accordance with \cite{KoTo16},  we simplify the notation from Section~\ref{review}  by dropping the index $a$ and noting that the normalized scalar curvature of $S= \C P^1$ is $1$.
Thus, \eqref{scalht} becomes
$$ Scal(h_t) = \frac{-(z+t)^2F''(z) + 6(z+t) F'(z) - 12 F(z) + 4x (z+t)^2}{(1+x z)}$$
whereas
the volume form of $h_t = (z+t)^{-2} g$ is
$$dv_{h_t}= (z+t)^{-4} (1/x+ z)\omega_{\C P^1}\wedge dz\wedge \theta.$$
Using that $f(t)$ is independent of  the choice of $F(z)= \Theta(z)p_c(z)$, we can take  $F_c(z) = (1-z^2)(1+xz)$ (see \eqref{g}), and compute
$$
f(t)  =  \frac{\int_M Scal(h_t)v_{h_t}}{\sqrt{\int_M dv_{h_t}}}=4\pi\sqrt{6}\frac{\left(1-2 x - 2 x t +(1+2x)t^2\right)}{\sqrt{x\left( 1-4 x t + 3 t^2\right)(t^2-1)}}
$$
and
$$f'(t) = 16\pi\sqrt{6}\frac{x(xt^2-2t+x)((1- x)t^2-xt+x)}{(x(t^2-1)(3t^2-4xt+1))^{3/2}}.$$
It is not hard to check that if $x\leq 4/5$, $t=\bb_0(x)$  with $\bb_0(x)$ defined by \eqref{a1-x} is the only critical point of $f(t)$ for $|t|>1$,  and  it is a minimum.  If, on the other hand,
$4/5<x<1$, then we find three critical points for $f(t)$:
$\bb_0(x)= \frac{1+\sqrt{1-x^2}}{x}$ , $\bb_{+}(x)= \frac{x+ \sqrt{x(5x-4)}}{2(1-x)}$, and $\bb_{-}(x)= \frac{x- \sqrt{x(5x-4)}}{2(1-x)}$
(all greater than 1 with $\bb_-(x) < \bb_0(x)<\bb_+(x)$).  As noticed in \cite{Le15}, the values $\bb_{\pm}(x)$  give rise to the same admissible K\"ahler metric, i.e. $F_{x, \bb_{+}(x), 4}(z)= F_{x, \bb_-(x),4}(z)$. We denote by $g_0$ the admissible K\"ahler metric corresponding to $(x, \bb_0(x))$, by $g$ the admissible K\"ahler metric corresponding to $(x, \bb_{\pm}(x))$, and by $h_0=\frac{1}{(z+ \bb_0(x))^2} g_0, h_{\pm} = \frac{1}{(z+\bb_{\pm}(x))^2} g$ the corresponding Einstein--Maxwell metrics. Even though the function $f(t)$ is the same for all of these cases, we are dealing with two different conformal structures,  $[g_0]$ and $[g]$.

Now, as  the only critical points of $f(t)$ for $t>1$ are $\bb_0(x), \bb_{\pm}(x)$, they cannot all occur as minima. Indeed,  it is not hard to see that $f(t)$ has a relative maximum at
$t=\bb_0(x)$ (and relative minima at $t=\bb_{\pm}(x)$).
Thus,  the conformally K\"ahler,  Einstein--Maxwell metric $h_0$ is not a Yamabe minimizer of $[g_0]$, and this
despite the fact that for $4/5<x<1$ we have
$$f(\bb_0(x))=4\pi\sqrt{6} \frac{(1+2x)\sqrt{1-x^2}+(1+2x-x^2)}{\sqrt{x(6-5x^2+(6-2x^2)\sqrt{1-x^2})}}  <8\pi\sqrt{6},$$
i.e. $h_0$  does not violate the estimate \eqref{aubin}.
\begin{remark} Otoba \cite{otoba} produced a different family of conformally K\"ahler,  constant scalar curvature metrics on all Hirzebruch surfaces,  and in particular proved that (on each Hirzebruch surface)
some of these are not Yamabe minimizers.
\end{remark}

\begin{remark}
If we return to the explicit example in Remark \ref{multiplesolnoverHodge4mnf}, where $S={\mathbb C}{\mathbb P}^2$,
$s=3$, and $x=4/5$, one may also observe that $ {\mathfrak S}(h_t)$ has a local maximum at $t=\bb_0(x)$ and local minima at $t=\bb_{\pm}(x)$. Further, ${\mathfrak S}(h_{0}) < 30\pi\left(\frac{16}{15}\right)^{\frac{1}{3}}$, so, similarly to the first Hirzebruch surface,  we have a constant scalar curvature  Einstein--Maxwell metric which is not  a Yamabe minimizer  but  satisfies \eqref{aubin}.
\end{remark}

\appendix
\section{{Orthotoric $(f,\p)$-extremal metrics}}\lb{orthog}

The bundle geometry examined so far in this paper is related to the theory of hamiltonian
$2$-forms of order $\ell=1$, see \cite{haml}. In this appendix, we will describe local examples of $(f,\p)$-extremal metrics which admit a hamiltonian $2$-form of order $\ell=m$, that is, orthotoric K\"ahler metrics. Presumably, similar explicit constructions hold for K\"ahler metrics admitting a hamiltonian $2$-form of any order $1\le \ell \le m$, but this will not investigated in this paper.

Recall from \cite{haml} that a K\"ahler $2m$-manifold $(M,g,J,\om)$ is orthotoric if it is equipped with
$m$ Poisson-commuting Killing potentials $\sig_1,\ldots,\sig_m$ such that
on a dense open set, the roots $\xi_j$ of $\sum_{r=0}^m(-1)^r\sig_rt^{m-r}$,
with $\sig_0=1$, are smooth, with linearly independent, orthogonal gradients. The functions $(\xi_1, \ldots, \xi_m)$ together with the angular coordinates $(t_1, \ldots, t_m)$ for the momenta $(\sigma_1, \ldots, \sigma_m)$ form a coordinate system  on  that open dense subset,  and are called {\it orthotoric coordinates}. It is shown in \cite{haml} that  with respect to the orthotoric coordinates
the metric is given   by
\begin{equation}\lb{ortho}
\begin{split}
g &=\sum_{j=1}^m\fr{\Delta_j}{\Theta_j(\xi_j)}d\xi_j^2+\sum_{j=1}^m\fr{\Theta_j(\xi_j)}{\Delta_j}
\left[\sum_{r=1}^m\sig_{r-1}(\wht{\xi}_j)dt_r\right]^2\!\!, \\
\omega &=  \sum_{j=1}^m d\xi_j \wedge \big(\sum_{r=1}^m\sig_{r-1}(\wht{\xi}_j)dt_r\Big),
\end{split}
\end{equation}
where $\Delta_j=\prod_{j\ne k} (\xi_j-\xi_k)$, each $\Theta_j(z)$ is a function of one variable,
and $\sig_{r-1}(\wht{\xi}_j)$ is the $r-1$-th elementary symmetric function of the remaining $\xi_k$'s
after $\xi_j$ is removed.

It is easily deduced  from \eqref{ortho} that the following formulas hold
\begin{equation}\label{fp-data}
\begin{aligned}
&g(\nab\sig_k,\nab\sig_l)=\textstyle{\sum}_{j=1}^{m}
\sig_{k-1}(\hat\xi_j)\sig_{l-1}(\hat\xi_j)\Theta_j(\xi_j)/\Delta_j,\\[3pt]
&\Delta_g \sig_k=-\textstyle{\sum}_{j=1}^m\,\sig_{k-1}\!(\hat{\xi}_j)\Theta_j'(\xi_j)/\Delta_j,\\[3pt]
&Scal(g)=-\textstyle{\sum}_{j=1}^m\Theta_j''(\xi_j)/\Delta_j,
\end{aligned}
\end{equation}
where $\nab \sig_k$, $\Delta_g \sig_k$ are the gradient and Laplacian of $\sig_k$, respectively, $Scal(g)$ is the scalar curvature of $g$,
and the primes denote differentiation with respect to $\xi_j$.
Indeed, the last two formulas are obtained from (78),(79) in \cite{haml}, after noting that $\partial \sig_k/\partial \xi_j=\sig_{k-1}(\hat{\xi}_j)$.
The first follows from (54) in \cite{haml},  by noting that
\begin{equation}
|\nab\xi_j|^2_g
=\Theta_j(\xi_j)/\Delta_j.
\end{equation}

We now recall equation \eqref{scalh1} in the form
\[
Scal_{f,\p}(g)= f^2 Scal(g)  -2(\p-1) f\Delta_g f - \p(\p-1)|\nab f|^2_g.
\]
We notice that $\sigma_k$ is a Killing potential for the Killing vector field $\partial/\partial t_k$ of \eqref{ortho}, and any Killing potential $f$  of a Killing vector field commuting with $\partial/\partial t_j, j=1, \ldots, m$ is necessarily an affine function in the $\sig_k$'s. Thus, for such an $f$, 
the above formulas show that
$g$ is $(f,\p)$-extremal if and only if there exist constants $a_k$, $b_k$, $k=0,\ldots,m$
such that
\begin{multline}\label{fp-ext}
-\Big(\sum_{k=0}^ma_k\sig_k\Big)^{\!2}\sum_{j=1}^m\Theta_j''(\xi_j)/\Delta_j
+2(\p-1)\Big(\sum_{k=0}^ma_k\sig_k\Big)
\!\!\sum_{k,j=1}^m a_k\sig_{k-1}(\hat\xi_j)\Theta_j'(\xi_j)/\Delta_j\\
-\p(\p-1)\!\!\sum_{k,l,j=1}^m\!\!\!\Big(a_ka_l
\sig_{k-1}(\hat\xi_j)\sig_{l-1}(\hat\xi_j)\Theta_j(\xi_j)/\Delta_j\Big)
=\sum_{k=0}^mb_m\sig_k.
\end{multline}

We now consider some special cases.

\subsection{Bochner-flat orthotric metrics are $(f, m+2)$-extremal}\label{s:BF} It has been observed in \cite{am} for $m=2$ and in \cite{ac} for $m\ge 2$ that a Bochner--flat metric is $(f, m+2)$-extremal for any positive Killing potential $f$. The metric \eqref{ortho} is Bochner-flat iff $\Theta_j(z)=P(z)$ are all equal to a $j$-independent polynomial of degree $\le m+2$, see \cite[Prop.~17]{haml} and the references therein. Thus, in this case, we get a solution of \eqref{fp-ext} with $\p=m+2$ for any choice of $a_0, \ldots, a_m$.  

\subsection{Flat orthotoric metrics which are $(f,\p)$-extremal} Recall from \cite[Prop.~17]{haml} that \eqref{ortho} is flat iff $\Theta_j(z)=P(z)$  for a ($j$-independent) polynomial $P$  of degree $\le m$. In this case, we  show 
\begin{prop}\label{p:flat} Let $(g, \omega)$ be a flat orthotoric metric in the form \eqref{ortho}  with $\Theta_j(z)= P(z)$  for a polynomial $P(z)$ of degree $\le m$,  and  $f= \sum_{r=0}^m a_r \sigma_r$ be a positive Killing potential.  Then,  $(g, \omega)$ is $(f, \p)$-extremal for any $\p$.
\end{prop}
\begin{proof} Using the Vandermonde identitity (see \cite[App.~A]{haml})
\begin{equation}\label{vander-general}
\sum_{j=1}^m \frac{\xi_j^{m-s}\sigma_{r-1}(\hat \xi_j)}{\Delta_j} = (-1)^{s-1}\delta_{rs}, \ r,s=1, \ldots, m,
\end{equation}
with $r=1$ yields that the first term in the LHS of \eqref{fp-ext} is identically zero when $\Theta_j(z)=P(z)$ for a polynomial $P$ of degree $\le m$. Similarly, \eqref{vander-general} shows that the second  term is an affine-linear function in $\sigma_1, \ldots, \sigma_m$.  We thus conclude that the metric $(g, \omega)$ is $(f, \p)$-extremal if $\p=0,1$, and for $\p \neq 0,1$ it is $(f, \p)$-extremal iff 
\begin{equation*}\sum_{k,l,j=1}^m \Big(a_ka_l
\sig_{k-1}(\hat\xi_j)\sig_{l-1}(\hat\xi_j)\frac{P(\xi_j)}{\Delta_j}\Big)
\end{equation*}
is an affine-linear function in $(\sigma_1, \ldots, \sigma_m)$. Notice that the latter condition  does not depend on $\p$, and it does hold for $\p=m+2$ by the discussion in Sect.~\ref{s:BF}. Thus, we conclude that $(g, \omega)$ is $(f, \p)$-extremal for any value of $\p$. \end{proof}
In some special cases of Proposition~\ref{p:flat}, we can find explicitly  the relationship between the coefficients $a_0, \ldots, a_m$ of $f$, those of $P(z)$ and $b_0, \ldots, b_m$ in the LHS.  For example, let us take $f= a_0 + a_1\sigma_1$ with $a_1\neq 0$ and  $\Theta_j(z)= P(z)= \sum_{k=0}^m c_k z^{m-k}$.   By  the Vandermonde identities (see \cite[App. B]{haml})
\begin{equation}\label{vandermonde}
\begin{split}
\sum_{j=1}^m \frac{\xi_j^{m-s}}{\Delta_j} &= \delta_{s1}, s=1, \ldots, m; \\
\sum_{j=1}^m \frac{\xi_j^{m}}{\Delta_j} & = \sigma_1.
\end{split}
\end{equation}
the LHS of \eqref{fp-ext} now reduces to
\begin{equation*}
\begin{split}
& -(a_0+a_1\sig_1)^2\textstyle{\sum}_{j=1}^m\frac{P''(\xi_j)}{\Delta_j }  + (\p-1)\!\!\left[2(a_0+a_1\sig_1)a_1\textstyle{\sum}_{j=1}^m\frac{P'(\xi_j)}{\Delta_j}
-\p a_1^2\textstyle{\sum}_{j=1}^m\frac{P(\xi_j)}{\Delta_j}\right] \\
& = (\p-1)\left[2m(a_0 + a_1\sigma_1)a_1c_0 - \p a_1^2c_0 \sigma_1 -\p a_1^2c_1\right] \\
&= b_0 + b_1\sigma_1
\end{split}
\end{equation*}
with
\begin{equation*}
b_1 = a_1^2c_0 (\p-1)(2m-\p), \ \ b_0= (\p-1)(2m a_0a_1c_0 - \p a_1^2c_1).
\end{equation*}
We conclude  in this case that  \eqref{ortho} is a flat K\"ahler metric which has constant  $(a_0+a_1\sigma_1, \p)$-scalar curvature iff $c_0=0$ or $\p=1, 2m$. In particular, we get an $(m+2)$-dimensional family (parametrized by $a_0, c_0, \ldots, c_m$) of  conformally-K\"ahler, Einstein--Maxwell metrics for which the K\"ahler metric is flat.  We also notice that,  more generally,   if $\p \neq 0,1$ the coefficients $b_0$ and $b_1$ uniquely determine the coefficients $c_0$ and $c_1$ of $P$, but not its other coefficients.

\subsection{Orthotoric metrics which are $(\sig_m, \p)$-extremal}
We now consider case $f=\sig_m=\prod_{i=1}^m\xi_i$.  In this case we will exhibit solutions extending some of the ambitoric examples  discussed in \cite{ac,am} to higher dimensions.
\begin{prop}\label{p:sigma-m}
Let $(g, \omega)$ be an orthotoric metric  of the form \eqref{ortho}. Then $g$ is  $(\sig_m,\p)$-extremal metric with $m\ge 2$ and $\p\ne 1,\ldots,m+1$ iff each function $\Theta_j(z)$  is a sum of a polynomial $P(z)$
of degree $\le m$,  whose coefficients are independent of $j$,  and
an expression of the form $b_{1j}z^{\p-1}+b_{2j}z^\p$,
for arbitrary constants $b_{ij}$, $i=1,2, j= 1, \ldots, m$. Furthermore, in this case
$Scal_{\sig_m,\p}(g)$ is a linear combination of $\sigma_{m-1}$ and $\sigma_{m}$, i.e. the coefficients
$b_i=0$ for $i<m-1$. In particular, $Scal_{\sig_m,\p}(g)$ is constant iff it vanishes, which happens iff $P(0)=P'(0)=0$.
\end{prop}
For  $\p \in \{1,\ldots,m+1\}$, we can find solutions $\Theta_j(\xi_j)$ of a similar form  but the first summand will contain logarithmic terms.
Note also that if at least one of the $b_{i\!j}$'s is nonzero, $g$ is not flat.

\begin{proof} For $f=\sigma_m$, Equation \Ref{fp-ext} becomes
\begin{multline}\lb{sigm0}
-\sigma_m^2\sum_{j=1}^m \Theta_j''(\xi_j)/\Delta_j\\
+(\p-1)\!\!\left[2\sigma_m
\textstyle{\sum}_{j=1}^m(\Theta_j'(\xi_j)/\Delta_j)\sig_{m-1}\!(\hat{\xi}_j)
-\p\textstyle{\sum}_{j=1}^m(\Theta_j(\xi_j)/\Delta_j)\sig_{m-1}^2\!(\hat{\xi}_j)\right]
=\sum_{k=0}^mb_k\sig_k,
\end{multline}
with $\sig_{m-1}(\hat{\xi}_j)=\sigma_m/\xi_j$.

We first check that  $\Theta_j(x)= P(z) + b_{1j}z^{\p-1}+b_{2j}z^\p$ as in the proposition give a solution. To this end, we re-write \eqref{sigm0} as
\begin{equation}\label{homog}
\sigma_m^2\sum_{j=1}^m \Big( \frac{-\xi_j^2\Theta_j''(\xi_j)  + 2(\p-1)\xi_j \Theta_j'(\xi_j) -\p(\p-1) \Theta_j(\xi_j)}{\xi_j^2 \Delta_j}\Big) =\sum_{k=0}^mb_k\sig_k.
\end{equation}
We notice that  for each $j$, the term $b_{1j}z^{\p-1}+b_{2j}z^\p$  is the solution of the homogeneous ODE  at the LHS of \eqref{homog}, so it is enough to compute it  with $\Theta_j(z)=P(z)$ being a $j$-independent polynomial of degree $\le m$.  To this end, we use the Vandermonde identities \eqref{vandermonde} and
\begin{equation}\label{vandermonde-1}
\begin{split}
\sum_{j=1}^m \frac{\xi_j^{s-2}}{\Delta_j} &= (-1)^{m-1} \frac{\delta_{s1}}{\sigma_m}, s=1, \ldots, m;  \\
\sum_{j=1}^m \frac{\xi_j^{-2}}{\Delta_j} &=  (-1)^{m-1}\frac{\sigma_{m-1}}{\sigma_{m}^2},
\end{split}
\end{equation}
which follow from \eqref{vandermonde}  written for $(1/\xi_i)$'s instead of the $\xi_i$'s.
The claim in the proposition follows easily from \eqref{vandermonde} and \eqref{vandermonde-1}.  This computation also shows that $Scal_{\sigma_m, \p}(g)$ is a linear combination of $\sigma_{m-1}$ and $\sigma_m$, i.e.   $b_0=\cdots =b_{m-2}=0$,  as well as the condition for the vanishing of $Scal_{\sigma_m, \p}(g)$. As mentioned above, by the classification in \cite{haml}, the orthotoric
metric $g$ is not flat  provided that at least one of the constants $b_{ij}$ is nonzero.

\smallskip
We now turn to the necessity of the conditions.
Denote by $\Delta= (-1)^{m(m-1)/2}\prod_{i<j}(\xi_i-\xi_j)$  the  Vandermont determinant of $\xi_1, \ldots, \xi_m$ and by $\Delta({\hat \xi}_j)$ the Vandermonde determinant of $\xi_k$ with $k\neq j$.  Notice that 
$\Delta({\hat \xi}_j)$ is, up to sign,  $\Delta/\Delta_j$.

We can now rewrite \Ref{sigm0} in the form
\begin{equation*}
\textstyle{\sum}_{j=1}^m\left[\pm\sig^2_{m-1}(\hat{\xi}_j)
\Delta(\hat \xi_j)
[-\xi_j^2\Theta_j''+2(\p-1)\xi_j\Theta_j'-\p(\p-1)\Theta_j]\right]
=\Delta \sum_{k=0}^mb_k\sig_k,
\end{equation*}
obtained by  multiplying both sides of \eqref{homog} by  $\Delta$ and
rearranging the result, with the signs $\pm$ left unspecified, as they will not matter
for the rest of the argument. This equation has the form
\begin{equation}\lb{form0}\sum_{j=1}^m F_j(\hat{\xi}_j)H_j(\xi_j)=G,\end{equation} where $H_j$
is the expression in the inner square brackets,  $F_j(\hat{\xi}_j)$ the rest of the $j$-th summand on the LHS (which does not depend on $\xi_j$), and $G$ a polynomial function of all the $\xi_j$'s.

We now count degrees. For $G$, we notice that the combination of $\sig_k$'s has all terms of degree at most one
in $\xi_j$, whereas  $\Delta$ has degree $m-1$ in each $\xi_j$, so $G$ has degree at
most $m$ in each $\xi_j$. Also,  $F_j(\hat{\xi}_j)$ has degree $2+m-2=m$ in each $\xi_k$,
$k\neq j$. Differentiating equation \eqref{form0} $m$ times with respect to, say, $\xi_1$,  yields
\[
 F_1(\hat{\xi}_1)H^{(m)}_1(\xi_1)+\sum_{j=2}^m
F^{(m)}_j\!(\hat{\xi}_1,\hat{\xi}_j)H_j(\xi_j)=k,
\]
Where `$(m)$' denotes this $m$-th partial derivative,
$F^{(m)}_j(\hat{\xi}_1,\hat{\xi}_j)$ does not depend on $\xi_1$ or $\xi_j$ and
$k$ is a constant.
Separation of variables yields that $H^{(m)}_1(\xi_1)$ is constant,
so that, $H_1$, and similarly each $H_j(\xi_j)$, $j=1,\ldots,m$, is a polynomial of degree at most $m$ in $\xi_j$. Now for $\p\ne 1,\ldots,m+1$, the solutions of
\begin{equation}\label{small-ode}
H(\xi_j)=-\xi_j^2\Theta_j''(\xi_j)+2(\p-1)\xi_j\Theta_j'(\xi_j)-\p(\p-1)\Theta_j(\xi_j)=
\textstyle{\sum}_{k=0}^m c_k\xi_j^k
\end{equation}
have the form $\Theta_j(\xi_j)=P_j(\xi_j)+b_{1j}\xi_j^{\p-1}+b_{2j}\xi_j^{\p},$
where $P_j(z)$ is  a polynomial of degree at most $m$.

To show the independence from $j$ of  the coefficients of $P_j(\xi_j)$, we
consider another form for equation \eqref{sigm0},
obtained by multiplying it by $\Delta_k$ for some fixed $k$.
\begin{multline*}
-\sig_m^2\Big(\Theta_k''+\textstyle{\sum}_{j\ne k}(\Delta_k/\Delta_j)\Theta_j''\Big)
+2(\p-1)\sig_m
\Big(\textstyle{\sig_{m-1}\!(\hat{\xi}_k)\Theta_k'
+\sum}_{j\ne k}^m((\Delta_k/\Delta_j)\sig_{m-1}\!(\hat{\xi}_j)\Theta_j'\Big)\\
-\p(\p-1)\Big(\sig_{m-1}^2\!(\hat{\xi}_k)\Theta_k+\textstyle{\sum}_{j\ne k}(\Delta_k/\Delta_j)\sig_{m-1}^2\!(\hat{\xi}_j)\Theta_j\Big)
=(\sum_{l=1}^mb_m\sig_l)\Delta_k,
\end{multline*}
Setting $\xi_k=\xi_{j_0}$ for some fixed $j=j_0\ne k$ in this equation,
we note that in a non-empty open set
$\Delta_k/\Delta_{j}\big|_{\xi_k=\xi_{j_0}}=-\delta_{jj_0}$,
so that we obtain
\begin{multline*}
-\sig_m^2\big|_{\xi_k=\xi_{j_0}}\Big(\Theta_k''(\xi_{j_0})-\Theta_{j_0}''(\xi_{j_0})\Big)
+2(\p-1)\sig_m\big|_{\xi_k=\xi_{j_0}}
\Big(\sig_{m-1}\!(\hat{\xi}_k)\Theta_k'(\xi_{j_0})
-\sig_{m-1}\!(\hat{\xi}_{j_0})\big|_{\xi_k=\xi_{j_0}}\Theta_{j_o}'(\xi_{j_0})\Big)\\
-\p(\p-1)\Big(\sig_{m-1}^2\!(\hat{\xi}_k)\Theta_k(\xi_{j_0})
-\sig_{m-1}^2\!(\hat{\xi}_{j_0})\big|_{\xi_k=\xi_{j_0}}\Theta_{j_0}(\xi_{j_0})\Big)
=0.
\end{multline*}
After dividing by the common factor $\sig_{m-1}^2\!(\hat{\xi}_{k})$ this simplifies to
\begin{multline*}
-\xi_{j_0}^2\Big(\Theta_k''(\xi_{j_0})-\Theta_{j_0}''(\xi_{j_0})\Big)
+2(\p-1)\xi_{j_0}
\Big(\Theta_k'(\xi_{j_0})
-\Theta_{j_o}'(\xi_{j_0})\Big)\\
-\p(\p-1)\Big(\Theta_k(\xi_{j_0})
-\Theta_{j_0}(\xi_{j_0})\Big)
=0.
\end{multline*}
Denoting the coefficients of $P_j(z)$ by $a_{j\ell}$, it follows from the known form
of $\Theta_j$ and the last equation that
\[
(a_{k\ell}-a_{j_0\ell})[-\p(\p-1)+2(\p-1)\ell-\ell(\ell-1)]=0,\qquad \ell=2\ldots m,
\]
Since the factor in the square brackets vanishes only when $\ell=\p$ or $\ell=\p-1$, and
$\p\ne 2,\ldots, m+1$, it follows that $a_{k\ell}=a_{j_0\ell}$, $\ell=2,\ldots, m$. \end{proof}

We notice that,   similarly to the example discussed after Proposition~\ref{p:flat}, the arguments in first part of the proof of Proposition~\ref{p:sigma-m} show that  $b_{m-1}$ and $b_m$ depend only on the affine part of $P(z)$ whereas its  other coefficients,   as well as  the real constants $b_{1j}$ and $b_{2j}$,  are arbitrary. This gives rise to a $(3m-1)$-dimensional family of orthotoric K\"ahler metrics of constant $(\sigma_m, \p)$-scalar curvature.

\end{document}